\documentclass[a4paper,11pt]{article}
\title{Fermat's Last Theorem and modular curves over real quadratic fields}
\author{Philippe Michaud-Jacobs}

\makeatletter
\newcommand\notsotiny{\@setfontsize\notsotiny\@vipt\@viipt}
\makeatother
\usepackage{amsmath} 
\usepackage{amssymb} 
\usepackage{mathtools}
\usepackage{framed}
\usepackage{pifont} 
\usepackage{enumerate} 
\usepackage{wasysym} 
\usepackage{commath}
\usepackage{graphicx}
\usepackage{framed}
\usepackage{amsmath}
\usepackage{amssymb}
\usepackage{amsfonts}
\usepackage{mathrsfs}
\usepackage{mathtools}
\usepackage{url}
\usepackage{array}
\usepackage{amsthm}
\usepackage[english]{babel}
\usepackage[utf8]{inputenc}
\newtheorem{theorem}{Theorem}[section]
\newtheorem*{theorem*}{Theorem}

\newtheorem {lemma}[theorem]{Lemma}
\newtheorem {proposition}[theorem]{Proposition}
\theoremstyle{definition}

\newtheorem {example}{Example}[section]

\theoremstyle{remark}
\newtheorem{remark}[theorem]{Remark}
\newtheorem*{remark*}{Remark}
\usepackage{yfonts}
\usepackage{amsmath,amssymb,tikz-cd}
\usepackage{pdflscape}

\usepackage{etoolbox}
\apptocmd{\sloppy}{\hbadness 10000\relax}{}{}

\usepackage[numbers]{natbib}
\usepackage{cite}
\usepackage{mathtools, nccmath}

\providecommand{\Q}{\mathbb{Q}}
\providecommand{\Z}{\mathbb{Z}}
\providecommand{\C}{\mathbb{C}}

\renewcommand{\arraystretch}{2.0}

\usepackage{longtable}
\usepackage{hyperref}

\newcommand{\Addresses}{{
  \bigskip
  \footnotesize

 \textsc{Mathematics Institute, University of Warwick, CV4 7AL, United Kingdom}\par\nopagebreak
  \textit{E-mail address}: \texttt{p.rodgers@warwick.ac.uk}
}}

\let\svthefootnote\thefootnote
\newcommand\freefootnote[1]{%
  \let\thefootnote\relax%
  \footnotetext{#1}%
  \let\thefootnote\svthefootnote%
}

\date{\vspace{-3ex}}

\begin{document}

\maketitle
 
\begin{abstract}
In this paper we study the Fermat equation $x^n+y^n=z^n$ over quadratic fields $\Q(\sqrt{d})$ for squarefree $d$ with $26 \leq d \leq 97$. By studying quadratic points on the modular curves $X_0(N)$, $d$-regular primes, and working with Hecke operators on spaces of Hilbert newforms, we extend work of Freitas and Siksek to show that for most squarefree $d$ in this range there are no non-trivial solutions to this equation for $n \geq 4$.
\end{abstract}

\tableofcontents

\section{Introduction}

There has been much recent interest in the study of the Fermat equation \[ x^n+y^n=z^n \] over number fields. Following in the footsteps of Wiles \citep{wiles}, we would ideally like to show that this equation has no non-trivial solutions for $n \geq 4$ and $x,y,z \in K$, a number field. By a \emph{non-trivial} solution, we mean $xyz \neq 0$. The study of the Fermat equation over number fields dates back to the work of Maillet in the late 19th century \citep[p.~578]{hist}. 

\freefootnote{\emph{Keywords}: Fermat's Last Theorem, Fermat equation, Frey curve, Galois representations, quadratic points, modular curves, irreducibility,  Hilbert modular forms.}
\freefootnote{\emph{MSC2010}: 11D41, 11F80, 11G18, 11G05, 14G05.}
\freefootnote{The author is supported by an EPSRC studentship and has previously used the name Philippe Michaud-Rodgers.}

Asymptotic versions of Fermat's Last Theorem over number fields (proving there are no non-trivial solutions if all the prime factors of $n$ are greater than some bound dependent on $K$) have been proven over various number fields (see \citep{asymptotic,asymp2} for example). In this paper, we are concerned with trying to prove the non-existence of non-trivial solutions for \emph{all} $n \geq 4$ over a real quadratic field $K$. Jarvis and Meekin proved this statement over the field $\Q(\sqrt{2})$ in \citep{jarv}. This work was then extended by Freitas and Siksek  \citep{realquad} to the quadratic fields $\Q(\sqrt{d})$, with $d$ squarefree, $d \neq 5,17$ in the range $3 \leq d \leq 23$, where it was shown that there are no non-trivial solutions for $n \geq 4$. Kraus, in \citep{krausreal}, proved that for $K$ a cubic field of discriminant $148, 404$, or $564$ there are no non-trivial solutions for $n \geq 4$. The aim of this paper is to extend the result of Freitas and Siksek to squarefree $d$ in the range $26 \leq d \leq 97$, as well as introduce techniques that can be used to study other Diophantine equations, both over the rationals and number fields of low degree. For most values of $d$ in this range, issues arise surrounding irreducibility of Galois representations and the computation of Hilbert newforms. In most cases, we overcome these issues to obtain the following result.
\begingroup
\renewcommand\thetheorem{1}
\begin{theorem}\label{Mainthm} The equation \[x^n+y^n=z^n, \quad x,y,z \in K, \] has no non-trivial solutions for $n \geq 4$ and $K = \Q(\sqrt{d})$, when $d \in \mathcal{D}$, where \linebreak $\mathcal{D} = \{ 
26, 29, 30, 31, 35, 37, 38, 42, 43, 46, 47, 51, 53, 58, 59, 61, 62, 65, 66, 67, \\ 69,  71,  73,  74,  77, 79, 82, 83, 85, 87, 91, 93, 94
 \}.$
\end{theorem}
\endgroup

We note that the case $d=79$ is proven in the paper of Freitas and Siksek \citep[p.~14]{realquad}. We also obtain the following partial results.

\begingroup
\renewcommand\thetheorem{2}
\begin{theorem}\label{thm2} The equation \[x^p+y^p=z^p, \quad x,y,z \in K, \] has no non-trivial solutions for $p \geq 5$ prime and $K = \Q(\sqrt{d})$, when $ d =34,55,86,97$ unless $(d,p) = (34,23), (55,23), (86,31),$ or $(97,17)$.
\end{theorem}
\endgroup

In this result, the leftover primes appear as exceptions as we are unable to discard certain Hilbert newforms (see Section 5.4). Furthermore, the primes $p$ are $d$-irregular in each case (see Section 6).

In some cases, extending the work of Kraus \citep{sqrt5}, we can obtain a lower bound on the size $p$.
\begingroup
\renewcommand\thetheorem{3}
\begin{theorem}\label{thm3} 
The equation \[x^p+y^p=z^p, \quad x,y,z \in K, \] has no non-trivial solutions for $5 \leq p \leq 10^7$, prime, and $K = \Q(\sqrt{d})$, when $d =17,33,41,57,89$. Moreover, if $d=89$, then we also have no non-trivial solutions when $p \geq 5$ and $p \equiv \pm 2 \pmod{5}$.
\end{theorem}
\endgroup

Four values of $d$, with $d$ squarefree and $26 \leq d \leq 97$, do not appear in the above theorems; namely $d = 39,70,78,$ and $95$. This is because the spaces of Hilbert newforms we considered for these values were too large to work with computationally (see Section 5.3). 

In order to prove Theorems \ref{Mainthm}, \ref{thm2}, and \ref{thm3}, one of the key results we use is Theorem \ref{thm4} below, which extends work of Najman and Turcas \citep[p.~2]{najturc}. For squarefree integers $N$ and $d'$, with $N>0$, we denote by $X_0^{d'}(N)$ the twist of the modular curve $X_0(N)$ over $\Q(\sqrt{d'})$ by the Atkin-Lehner involution $w_N$. A rational point on $X_0^{d'}(N)$ can be identified with a pair of Galois conjugate quadratic points on $X_0(N)(\Q(\sqrt{d'}))$ that are interchanged by $w_N$.

\begingroup
\renewcommand\thetheorem{4}
\begin{theorem}\label{thm4}
Let $p>19$, $p \ne 37$ be a prime. Let $E'$ be an elliptic curve defined over any quadratic field $K'= \Q(\sqrt{d'})$. Let $q>5$, $q \neq p$ be a rational prime such that each prime of $K'$ above $q$ is of multiplicative reduction for $E'$. Then $\overline{\rho}_{E',p}$ is irreducible if one of the following two conditions holds:
\begin{enumerate}[(i)]
\item the prime $q$ does not split in $K'$;
\item the prime $q$ splits in $K'$ and $X_0^{d'}(p)(\Q)= \emptyset$. 
\end{enumerate}
\end{theorem}
\endgroup

We note that $K'$ may be an imaginary quadratic field in the statement of this theorem. Using the work of Ozman \citep{ozmantwist}, we can often show that $X_0^{d'}(p)(\Q) = \emptyset$ for given $p$ and $d'$ (see Theorem \ref{twistthm} of this paper).

We now outline the rest of the paper. In Section 2 we overview the general proof strategy and state the key properties of the Frey elliptic curve, mainly following \citep{realquad}. In Sections 3 and 4 we study the question of irreducibility. In Section 3 we use techniques from class field theory and various irreducibility criteria to do this, and in Section $4$ we deal with the remaining cases by studying quadratic points on the modular curves $X_0(N)$, using local obstructions, formal immersions, relative symmetric Chabauty, modular parametrisations, and sieving techniques. In Section $5$ we work with Hecke operators to partially reconstruct and subsequently eliminate Hilbert newforms. Finally, in Section 6 we consider $d$-regular primes.

\bigskip

The \texttt{Magma} \citep{magma} and \texttt{SageMath} \citep{sagemath} code used to support the computations in this paper can be found at:

\begin{center}
\url{https://github.com/michaud-jacobs/flt-quad}
\end{center}

\bigskip

I would like to express my sincere gratitude to my supervisors Samir Siksek and Damiano Testa for many useful discussions and their support in writing this paper. I would also like to thank Filip Najman for some helpful comments. Finally, I would like to thank the anonymous referee for a careful reading of the paper.

\section{The Frey Curve}

In this section we provide a brief overview of how one associates an elliptic curve to a putative solution of the Fermat equation, and we state some properties of this curve. We then state the level-lowering theorem used to (hopefully) obtain a contradiction, proving the non-existence of non-trivial solutions. In this section, we mainly follow \citep[pp.~4-8]{realquad}.

We start by fixing some notation. Let $K$ be a real quadratic field $\Q(\sqrt{d})$ for some squarefree $d$ with $26 \leq d \leq 97$.  We will also use the notation $K' = \Q(\sqrt{d'})$ to denote a general (possibly imaginary) quadratic field. Let $\epsilon$ be a fundamental unit for $K$. Write $\mathrm{Cl}(K)$ for the class group of $K$. We denote a prime of $K$ by $\mathfrak{q}$. Write $S = \{ \mathfrak{q} :  \mathfrak{q} \mid 2 \}$, which consists of a single prime if $2$ is inert or ramifies in $K$, and two primes if $2$ splits in $K$.

In all cases we consider, $\mathrm{Cl}(K)$ has order $1$, $2$, $3$, or $4$, and the quotient group $\mathrm{Cl}(K) / \mathrm{Cl}(K)^2 $ has size $r = 1$ or $2$. If $r=2$ then we choose $\mathfrak{m}$ an odd prime ideal such that $[\mathfrak{m} ] $ represents the non-trivial element of $\mathrm{Cl}(K) / \mathrm{Cl}(K)^2 $. The table in the appendix displays our choice for the ideal $\mathfrak{m}$ in each case.

Although Theorem \ref{Mainthm} is stated for any $n \geq 4$ and for $x,y,z \in K$, we can reduce (as in the case of Fermat's Last Theorem over $\Q$) to the case of $n=p$, prime, and $x,y,z \in \mathcal{O}_K$. We therefore study the equation, which we refer to as the \emph{Fermat equation with exponent $p$}, \begin{equation*}\label{Maineqn} a^p+b^p+c^p=0, \qquad a,b,c  \in \mathcal{O}_K, ~ p \geq 5. \end{equation*}
As discussed in \citep[p.~2]{realquad}, it has been shown that this equation has no non-trivial solutions for primes $5 \leq p \leq 13$ (as well as for the same equation with exponents $n=4, 6$, and $9$) over real quadratic fields, and so we will assume that we have a non-trivial solution $(a,b,c)$ with $p \geq 17$. To this solution we associate the Frey elliptic curve \[E_{a,b,c,p}: Y^2 = X(X-a^p)(X+b^p), \] and write $E = E_{a,b,c,p}$. We write $\mathcal{N}$ for the conductor of $E$ which is an ideal in $\mathcal{O}_K$. 

\begin{lemma}[Frey curve invariants, {\citep[p.~5]{realquad}}] The curve $E$ has the following invariants \begin{align*} c_4 & = 16(a^{2p}-b^pc^p),  & c_6 & = -32(a^p-b^p)(a^p-c^p)(b^p-c^p),  \\
\Delta & = 16a^{2p}b^{2p}c^{2p},  & j & = c_4^3/ \Delta. 
\end{align*} \end{lemma} 

\begin{lemma}[Reduction of the Frey curve, {\citep[pp.~5-7]{realquad}}]\label{reduc} The curve $E$ has additive reduction at the primes in $S$ and also at $\mathfrak{m}$ when $r=2$. If $2$ splits or ramifies in $K$, then $E$ has potentially multiplicative reduction at the primes above $2$.
If $\mathfrak{q} \notin S$, and $\mathfrak{q} \neq \mathfrak{m}$ (when $r=2$), then $E$ is semistable at $\mathfrak{q}$ and $p \mid v_\mathfrak{q}(\Delta_\mathfrak{q})$, where $\Delta_\mathfrak{q}$ is the minimal discriminant at $\mathfrak{q}$. 
\end{lemma}

Even if we know that a prime $\mathfrak{q}$ is of semistable reduction for $E$, we will not know whether it is of good or multiplicative reduction. The following result gives a way of producing, for a fixed prime $p$, a prime of multiplicative reduction for $E=E_{a,b,c,p}$.

\begin{lemma}[Kraus {\citep[p.~9]{sqrt5}}]\label{krauslem} Let $p \geq 17$ be a prime and suppose there exists a natural number $n$  satisfying the following conditions: \begin{itemize}
\item we have $q:=np+1$ is a prime that splits in $\mathcal{O}_K$;
\item we have $q \nmid \mathrm{Res}(X^n-1,(X+1)^n-1)$.
\end{itemize}
Then both primes of $K$ above $q$ are of multiplicative reduction for $E$.
\end{lemma}

\begin{proof} This result is proven in the case $K = \Q(\sqrt{5})$ in \citep[pp.~9-10]{sqrt5}. The proof immediately generalises to $\Q(\sqrt{d})$. \end{proof}

We would like to level-lower the Frey curve $E$. We introduce the following notation. Let $\mathfrak{f}$ be a Hilbert eigenform of parallel weight $2$. We write $\Q_\mathfrak{f}$ for its Hecke eigenfield: the field generated by its eigenvalues under the Hecke operators. When $\mathfrak{f}$ is irrational (i.e. $\Q_\mathfrak{f} \neq \Q$), we write $\varpi$ for a prime above the rational prime $p$.

\begin{theorem}[Level-Lowering, {\citep[p.~4]{realquad}}]\label{levellower} Let $p \geq 17$ and suppose $\overline{\rho}_{E,p}$ is irreducible. Define \[ \mathcal{N}_p = \mathcal{N}/  \prod_{\substack{ {\mathfrak{q} \mid \mid \mathcal{N} }  \\ p \mid v_\mathfrak{q}(\Delta_\mathfrak{q}) }} \mathfrak{q}.  \] Then we can \emph{level-lower} E. That is, there exists a Hilbert newform $\mathfrak{f}$ at level $\mathcal{N}_p$ such that $\overline{\rho}_{E,p} \sim \overline{\rho}_{\mathfrak{f},\varpi}$, for some prime $\varpi$ of $\Q_\mathfrak{f}$ above $p$.
\end{theorem}

\begin{proof} As $E$ is defined over a real quadratic field, it is modular \citep{FlHS}. Lemma \ref{reduc} gives the other conditions needed to level lower $E$, other than irreducibility which is assumed in the statement of the theorem.
\end{proof}

We then need to calculate the various possibilities for the level $\mathcal{N}_p$ obtained by scaling $(a,b,c)$. We use the method described in \citep[pp.~4-8]{realquad} to obtain a list of possibilities for  $\mathcal{N}_p$. The table in the appendix displays the possible levels $\mathcal{N}_p$ we obtained. Our code produces the same data as in \citep[p.~9]{realquad} for $d<26$.
At this point, there are three main issues we need to overcome. 

\begin{enumerate}
\item  Proving irreducibility of $\overline{\rho}_{E,p}$.
\item Calculating the Hilbert newforms at each level $\mathcal{N}_p$.
\item Eliminating the Hilbert newforms at each level $\mathcal{N}_p$.
\end{enumerate}

Sections 3 and 4 are devoted to proving irreducibility. Section 5 is then concerned with calculating and eliminating newforms.

\section{Irreducibility I: Obtaining a Bound}

In order to level-lower our Frey curve (Theorem \ref{levellower}), we need irreducibility of the mod-$p$ Galois representation $\overline{\rho}_{E,p}$. 

\begin{theorem}\label{irred} The representation $\overline{\rho}_{E,p}$ is irreducible for all squarefree $d$ with $26  \leq d \leq 97$  and $p \geq 17$. 
\end{theorem}

In this section, we reduce the problem to only needing to deal with finitely many primes $p$. We will in fact often obtain the full irreducibility statement we need.

As before, $E=E_{a,b,c,p}$ denotes our Frey curve defined over the quadratic field $K=\Q(\sqrt{d})$, and we suppose for a contradiction that the mod-$p$ Galois representation $\overline{\rho}_{E,p}$ is reducible. Then \[ \overline{\rho}_{E,p} \sim \left( \begin{smallmatrix} 
\theta & * \\
0 & \theta'
\end{smallmatrix} \right), \] where $\theta, \theta' : G_K \rightarrow \mathbb{F}_p$ are the \emph{isogeny characters} of the elliptic curve at $p$ and satisfy $\theta \theta' = \chi_p$, where $\chi_p$ is the mod-$p$ cyclotomic character. We can interchange $\theta$ and $\theta'$ by replacing $E$ with an isogenous curve. The characters $\theta$ and $\theta'$ are unramified away from $p$ and the additive primes for $E$. The additive primes for $E$ are the primes of $K$ above $2$, and $\mathfrak{m}$ when $r=2$ (Lemma \ref{reduc}). We write $\mathcal{N}_{\theta}$ and  $\mathcal{N}_{\theta'}$ for the conductors of $\theta$ and $\theta'$ respectively. For $\mathfrak{q}$ an additive prime of $E$ with $\mathfrak{q} \nmid p$, we have that $v_\mathfrak{q}(\mathcal{N})$ is even, and $v_\mathfrak{q}(\mathcal{N}_\theta) = v_\mathfrak{q}(\mathcal{N}_{\theta'}) = v_\mathfrak{q}(\mathcal{N})/2$. So if $p$ is coprime to $\mathcal{N}_\theta$ then from our list of possibilities for $\mathcal{N}_p$ we obtain a list of possibilities for $\mathcal{N}_\theta$, since $\mathcal{N}_\theta$ is the square root of the additive part of $\mathcal{N}$. \citep[p.~10]{realquad}.

We now consider two cases. The first is when $p$ is coprime to one of $\mathcal{N}_\theta$ and $\mathcal{N}_{\theta'}$. By interchanging $\theta$ and $\theta'$ we may assume $p$ is coprime to $\mathcal{N}_\theta$.

\begin{lemma}\label{rcg} Suppose $p$ is coprime to $\mathcal{N}_{\theta}$. Let $G$ be the ray class group for the modulus $\mathcal{N}_{\theta} \infty_1 \infty_2$, where $\infty_1, \infty_2$ denote the two real places of $K$. Let $n$ denote the exponent of $G$. Then $E$ has a point of order $p$ defined over a number field $L$ of degree $m$, where $m=n$ if $n$ is even, and $m=2n$ if $n$ is odd. 
\end{lemma}

\begin{proof} (See \citep[p.~10]{realquad}). The order of $\theta$ divides $n$. We may assume that $\theta$ has order $n$, as otherwise we can reduce to a case where the exponent of $G$ is less than $n$. Suppose $n$ is odd. As $\theta$ has order $n$ it cuts out a field extension of degree $n$ over $K$, which we denote $L$, such that $\theta |_{G_{L}} = 1$. This gives a point of order $p$ over $L$, and $L$ is a number field of degree $2n$.

Suppose now that $n$ is even. Let $L$ be the field cut out by $\theta^2$. As $\theta^2$ has order $n/2$, $L$ is a number field of degree $2 \cdot n/2 = n$. Then $\theta |_{G_L}$ has order $2$, and so twisting by $\theta$ gives a point of order $p$ defined over $L$.
\end{proof}

This lemma then combines with the following classification of $p$-torsion of elliptic curves defined over number fields, obtained by studying points on the $n$th symmetric power of $X_1(p)$.

\begin{theorem}[{\citep[pp.~1-2]{etors}}]\label{tors} let $L$ be a number field of degree $n$. Let $E'/L$ be an elliptic curve with a point of order $p$ over $L$. We have that \begin{align*} \text{ if } n = 2 & \text{ then } p \leq 13; \\
\text{ if } n = 3 & \text{ then } p \leq 13 ; \\
\text{ if } n = 4 & \text{ then } p \leq 17; \\
\text{ if } n = 5 & \text{ then } p \leq 19 ; \\
\text{ if } n=6 & \text{ then } p \leq 19 \text{ or } p=37; \\
\text{ if } n = 7 & \text{ then } p \leq 23. 
\end{align*}
Moreover, if $E'$ has a point of order $37$ defined over a field $L$ of degree $6$, then $j(E') = -9317$.
\end{theorem}

We obtain a ray class group with exponent $6$, and consequently a point on $E$ defined over a number field of degree $6$, in the cases $d=37$ and $d=79$. In each case, the elliptic curves with $j$-invariant $-9317$ do not have full two-torsion over $\Q(\sqrt{d})$, and so do not arise from the Frey curve $E$.

When $n \geq 8$, one possibility is using the following bound of Oesterl\'e.

\begin{theorem}[Oesterl\'e, {\citep[p.~21]{etors}}]\label{oest} let $L$ be a number field of degree $n$. Let $E'/L$ be an elliptic curve with a point of order $p$ over $L$. Then $p \leq (3^{n/2}+1)^2$.
\end{theorem}

In most cases we consider, the exponents of the ray class groups are $\leq 4$, and so we obtain an excellent bound on $p$ right away. The cases where an exponent is $6$ are discussed above. However, for certain cases, namely when $d = 26,34,35,39,55,82,91,$ or $95$, a ray class group has exponent $8$ (see the table in the appendix) and so we cannot apply Theorem \ref{tors}. Applying Oesterl\'e's bound gives $p \leq (3^4+1)^2= 6724$, which is rather large. Instead we use the following strategy to obtain a better bound on $p$. The idea revolves around the following result, which builds on work of \citep[p.~2]{najturc}.

\begin{theorem}\label{mult} Let $E'$ be an elliptic curve defined over any quadratic field $K'$. Let $p$ be a prime, and suppose $\overline{\rho}_{E',p}$ is reducible. Let $q>5$, $q \neq p$ be a rational prime that does not split in $K'$, such that the unique prime of $K'$ above $q$ is of multiplicative reduction for $E'$. Then $p \leq 19$ or $p = 37$.
\end{theorem}

We prove this result in Section 4.2 once we have discussed modular curves and formal immersions.

We fix the following notation (which will also be used in Section $5$). For $\mathfrak{q}$, a prime of $K$, write $n_\mathfrak{q}$ for the norm of $\mathfrak{q}$, and define  \[ \mathcal{A}_\mathfrak{q} := \{ a \in \mathbb{Z} : \abs{a} \leq 2 \sqrt{n_\mathfrak{q}}, \quad  n_\mathfrak{q}+1 - a \equiv 0 \pmod{4}  \}. \] If $\mathfrak{q}$ is a prime of good reduction for $E$, we know that $a_{\mathfrak{q}}(E) \in \mathcal{A}_\mathfrak{q}$. This follows from the Hasse--Weil bounds and the fact that $E$ has full two-torsion over $K$. We then define, for $a \in \mathcal{A}_\mathfrak{q}$, \[ P_{\mathfrak{q},a} := X^2-aX + n_\mathfrak{q}. \] When $a = a_{\mathfrak{q}}(E)$, this is the characteristic polynomial of Frobenius at $\mathfrak{q}$.

\begin{proposition}\label{goodred} Suppose $\overline{\rho}_{E,p}$ is reducible with $p \geq 17$. Define $B := \mathrm{Norm}(\varepsilon^{12}-1)$. Let $\mathfrak{q}$  be a prime of $K$ above $q$, with $\mathfrak{q} \nmid 2,3,5,p, \mathfrak{m}$, and such that $q$ does not split in $K$. Define $r_\mathfrak{q} = 1$ if $\mathfrak{q}$ is a principal ideal, and $r_\mathfrak{q} =2$ otherwise. Define \[ R_q:= \mathrm{lcm}  \{ \mathrm{Res}(P_{\mathfrak{q},a}(X),X^{12r_\mathfrak{q}}-1) : a \in \mathcal{A}_\mathfrak{q} \}, \] where Res denotes the resultant of the two polynomials. 
Then \[ p \mid \Delta_K \cdot B \cdot R_\mathfrak{q} \quad \text{or} \quad p \in \{17,19,37\}. \]
\end{proposition}

We can then choose a set of primes $\{ \mathfrak{q}_1, \cdots, \mathfrak{q}_t \},$ and let $R:= \gcd \{ R_{\mathfrak{q}_i} \}$, so that either $p \mid \Delta_K \cdot B \cdot R$ or $p \in \{17,19,37 \}$. For the cases we considered, we found this to give much better results than applying Oesterl\'e's bound for $n=8$.

\begin{proof}
Let $\mathfrak{q}$ be a prime as defined in the proposition. Then by Lemma \ref{reduc}, $\mathfrak{q}$ is a prime of semistable reduction for $E$. Suppose $\mathfrak{q}$ is a prime of good reduction. We then apply the result of \citep[Theorem 1]{crit}: the possible non-constant isogeny characters are $\{12,0 \}$ and $\{ 0,12 \}$, and we obtain $B = \mathrm{Norm}(\epsilon^{12}-1)$. It follows that if $ p \nmid B \cdot \Delta_K$ then \[ p \mid  \mathrm{Res}(P_{\mathfrak{q},a_\mathfrak{q}(E)}(X),X^{12r_\mathfrak{q}}-1).\] We do not know $a_\mathfrak{q}(E)$, but as discussed above, we know it lies in $\mathcal{A}_\mathfrak{q}$, and so $p \mid R_\mathfrak{q}$ as defined in the proposition.

If instead, $\mathfrak{q}$ is a prime of multiplicative reduction for $E$, then we apply Theorem \ref{mult} to conclude that $p = 17, 19,$ or $37$. 
\end{proof}

We will now consider the cases where $p$ is neither coprime to $\mathcal{N}_\theta$ nor $\mathcal{N}_{\theta'}$. In these cases, $p$ must either split or ramify in $K$ (see \citep[p.~247]{semistables}).

\begin{lemma} Suppose $p$ neither coprime to $\mathcal{N}_\theta$ nor $\mathcal{N}_{\theta'}$ and that $p$ ramifies in $K$. Suppose $2$ is not inert in $K$. Let $n$ be the exponent of the ray class group modulo $\sqrt{\mathcal{N}_A}\infty_1 \infty_2$, where $\mathcal{N}_A$ is the additive part of the conductor $\mathcal{N}$. Then $p \mid 2^n-1$.
\end{lemma}

\begin{proof} As $2$ is not inert in $K$, any prime dividing $2$ is a prime of potentially multiplicative reduction for $E$ (Lemma \ref{reduc}). We then follow the argument of \citep[p.~10]{realquad} and apply \citep[Proposition 6.2]{realquad}, which is stated for $K=\Q(\sqrt{p})$ but also holds for $K=\Q(\sqrt{d})$ with $p$ ramifying in $K$.
\end{proof}

\begin{lemma}[{\citep[p.~11]{realquad}}]  Suppose $p$ neither coprime to $\mathcal{N}_\theta$ nor $\mathcal{N}_{\theta'}$ and that $p$ splits in $K$. Set $n=6$ if $2$ is inert in $K$, and set $n=2$ otherwise. Then $p \mid \mathrm{Norm}(\varepsilon^n-1)$ if $\varepsilon$ or $-\varepsilon$ is totally positive, and otherwise  $p \mid \mathrm{Norm}(\varepsilon^{2n}-1)$.
\end{lemma}

Combining the results stated in this section reduces the problem to dealing with a finite number of primes in each case. These are shown in Table 1. We have not included the primes $17$ and $19$ in the table, as we will see at the start of Section 4 that $\overline{\rho}_{E,p}$ is irreducible over all real quadratic fields when $p=17$ or $19$.

\begingroup
\begin{table}[ht!]\label{Tab3}
\begin{center}
\notsotiny
\begin{tabular}{|c||c|c|c|c|}
\hline 
$d$ & $26$ & $29$ & $34$ & $35$  \\
\hline
$p \geq 23$ & $37, 101, 103$  & $29$ & $23, 37, 59, 71, 83$ & $37,47,61,97$  \\
\hline
\hline 
$d$  & $37$ & $39$ & $53$ & $55$  \\ \hline $p \geq 23$  & $37$&  $37, 227$ & $53$ & $37,59, 89, 179, 2437$ \\ 
\hline
\hline
$d$  &  $59$ & $61$ & $69$ & $71$  \\
\hline
$p \geq 23$ & $23$ & $61,127$ & $23$ & $59$   \\ \hline
\hline 
$d$ & $73$ & $74$ & $82$ & $89$    \\ \hline $p \geq 23$ & $89$ & $43$ & $37,41, 109$ & $53$    \\
\hline
\hline 
$d$  & $91$ & $93$ & $94$  & $95$ \\ \hline
$ p \geq 23$ & $37, 47, 67, 787, 1049$ & $31$ & $151$ & $31,37,61,79,97,2027$ \\ 
\hline
\hline 
$d$ & $97$ & & & \\ \hline
$ p \geq 23$ & $467$ & & & \\ 
\hline

\end{tabular}
\caption{Irreducibility Step 1}
\end{center}
\end{table}
\endgroup
\normalsize

\section{Irreducibility II: Modular Curves}

In the previous section we saw how to go from proving irreducibility for all primes $p \geq 17$ to a finite (and often empty) subset of primes (see Table 1). In this section, we study quadratic points on certain modular curves $X_0(N)$ to complete the proof of Theorem \ref{irred}.

\subsection{Quadratic Points on Modular Curves}

Suppose $\overline{\rho}_{E,p}$ is reducible for $p = 17,19$, or $p \geq 23$ appearing in Table 1. As $E$ has full two-torsion over $K$, it gives rise to a non-cuspidal $K$-point on the modular curves $X_0(p)$, $X_0(2p)$, and $X_0(4p)$. It is therefore enough to show that one of $X_0(p)(K)$, $X_0(2p)(K)$, or $X_0(4p)(K)$ has no points that could arise from $E$. We write $g(X_0(N))$ for the genus of $X_0(N)$. We write $X_0^{(2)}(N)$ for the symmetric square of $X_0(N)$, and denote its points by pairs: $(y,z)$. A pair of quadratic points on $X_0(N)$ corresponds to a rational point on $X_0^{(2)}(N)$, and we will use this point of view in Section 4.2.

Recent works \citep{ozmansiksek,hyperquad,joshaquad} have studied quadratic points on $X_0(N)$ of genus $2 \leq g \leq 5$, as well as the genus $6$ hyperelliptic curve $X_0(71)$. Here, all quadratic points are considered, rather than working over a fixed quadratic field as we wish to do. We note that extending these results to (non-hyperelliptic) curves of genus $\geq 6$ quickly becomes computationally impractical, and so we do not seek to do this here. There are two basic cases: either $X_0(N)$ has finitely many quadratic points or infinitely many quadratic points. A curve of genus $\geq 2$ will have infinitely many quadratic points if and only if it is hyperelliptic, or bielliptic with bielliptic quotient an elliptic curve of rank $\geq 1$ (see \citep[p.~352]{HandS}). Pulling back points on the quotient gives rise to infinitely many quadratic points on the original curve. We call these points \emph{non-exceptional}, and points which do not arise in this way are said to be \emph{exceptional}. 

\begin{theorem}[Ogg {\citep[p.~451]{Ogg}}, Bars {\citep[p.~11]{biellipticbars}}]\label{oggbars} The curve $X_0(N)$ is hyperelliptic of genus $\geq 2$ if and only if $N \in \{ 22,23,26,28,29,30,31,33,35,37, \\ 39, 40,41, 46, 47,48,50,59,71 \}$. Furthermore, the hyperelliptic involution is of Atkin--Lehner type, unless $N=37,40,$ or $48$. 

The curve $X_0(N)$ is bielliptic with an elliptic quotient of positive rank if and only if $N \in \{37, 43, 53, 61, 65, 79,  83, 89, 101, 131\}$. 
\end{theorem}

We note that the curve $X_0(37)$ is both hyperelliptic and bielliptic, with an elliptic quotient of positive rank. Consequently, we do not use the terms exceptional and non-exceptional for quadratic points on this curve. For $X_0(N)$ non-hyperelliptic, the degree $2$ elliptic quotients of $X_0(N)$ with infinitely many rational points are all of the form $X_0^+(N) = X_0(N)/ \langle w_N \rangle$, where $w_N$ denotes the Atkin--Lehner involution corresponding to $N$. The papers \citep{joshaquad, ozmansiksek} classify all quadratic points in the cases where there are finitely many such points and $ 2 \leq g \leq 5 $. The values $N$ with $N$ of the form $p, 2p,$ or $4p$ with $p \geq 17$ in this list are $N = 34$, $38$, and $67$. We also make use of the classification of quadratic points on $X_0(62)$ given in \citep[p.~5]{NajVuk}. In these cases, all quadratic points are defined over imaginary quadratic fields, and so $\overline{\rho}_{E,p}$ is irreducible for $p \in \{17,19,31,67\}$ over all real quadratic fields. 

When the curve $X_0(N)$ has infinitely many points and $2 \leq g \leq 5$, or $X_0(N)$ is hyperelliptic, the papers \citep{joshaquad, hyperquad} obtain a classification of its exceptional quadratic points. In each case, no exceptional points could arise from $E$: they are either defined over imaginary quadratic fields or over real quadratic fields not appearing in Table 1. 

\begin{lemma}
Suppose $X_0(N)$ is hyperelliptic with $N \neq 37$, and let $P \in X_0(N)(\Q(\sqrt{d}))$ be a non-exceptional quadratic point. Then $P$ corresponds to a rational point on the quadratic twist of $X_0(N)$ by $d$, which we denote by $X_0^d(N)$. 
\end{lemma}

\begin{proof} We can take a model for the hyperelliptic curve $X_0(N)$ of the form $y^2 = f(x)$, with the hyperelliptic involution given by $(x,y) \mapsto (x,-y)$. If  $P = (u,v) \in X_0(N)(\Q(\sqrt{d}))$ is a non-exceptional quadratic point, then $(u,v)=(u^\sigma,-v^\sigma)$, where $\sigma$ generates $\mathrm{Gal}(\Q(\sqrt{d})/\Q)$. So $u \in \Q$ and $v= b\sqrt{d}$ with $b \in \Q$. So $db^2 = f(u)$, so $(u,b) \in X_0^d(N)(\Q)$.
\end{proof}

It follows that if $X_0^d(N)(\Q) = \emptyset$, then we have a contradiction. This is often easily checked by seeing whether or not we have points everywhere locally; we see this in Example \ref{extwist} below. Even if $X_0^d(N)$ has points everywhere locally, in some cases we can prove that $X_0^d(N)(\Q) = \emptyset$ using the \texttt{TwoCoverDescent} function in \texttt{Magma}. We do this, for example, in the case $d = 71$ and $p = 59$. This idea of twisting is taken further in \citep{ozmantwist}: for $X_0(N)$, possibly non-hyperelliptic, we denote by $X_0^d(N)$ the twist of $X_0(N)$ by $w_N$ over $\Q(\sqrt{d})$. As mentioned in the introduction, rational points on $X_0^d(N)$ correspond to pairs of Galois conjugate quadratic points on $X_0(N)(\Q(\sqrt{d}))$ that map to rational points in the quotient $X_0^+(N)$. We then have the following result. The full theorem is stronger than the version we present here.

\begin{theorem}[Ozman, {\citep[p.~2]{ozmantwist}}]\label{twistthm} Let $N>0$ be squarefree,  $K' = \Q(\sqrt{d'})$ a quadratic field, and let $l \nmid N$ be a prime that is ramified $K'$. Suppose there exists a prime, $\mathfrak{l}$, of $M:=\Q(\sqrt{-N})$ above $l$, which is not a principal ideal. Then  $X_0^{d'}(N)(\Q_l) = \emptyset$, so $X_0^{d'}(N)(\Q) = \emptyset$.
\end{theorem}

\begin{remark} The condition `$\mathfrak{l}$ is not principal' in Theorem \ref{twistthm} is equivalently stated in {\citep[p.~18]{ozmantwist}} as `$\mathfrak{l}$ is not totally split in the Hilbert class field of $M$'.
\end{remark}

\begin{example}\label{extwist} We consider the hyperelliptic curve $X_0(31)$. A model for this curve is given by \[ y^2 = x^6 -8x^5 + 6x^4 + 18x^3 - 11x^2 - 14x - 3. \] Let $d=95$. We can check directly that $X_0^d(31)$ has no points over $\Q_5$. Alternatively, we can apply Theorem \ref{twistthm} with $l=5$. The prime $l$ ramifies in $\Q(\sqrt{d})$, and $l \nmid 31$. Write $M=\Q(\sqrt{-31})$ and $H$ for its Hilbert class field, which is a degree $3$ extension of $M$ since $M$ has class number $3$. The prime $l=5$ splits in $\mathcal{O}_M$, and the two primes above $5$ are not principal ideals in $\mathcal{O}_M$ (equivalently, they are not totally split in $H$). It follows that $X_0^d(31)(\Q_5) = \emptyset$. We deduce that for $d=95$, $\overline{\rho}_{E,31}$ is irreducible.
\end{example}

\subsection{Formal Immersions and Relative Symmetric Chabauty}

We start by reviewing some properties of the Jacobian, $J_0(N)$, of the modular curve $X_0(N)$. We then provide a proof of Theorem \ref{mult}, and see how similar ideas can be used to obtain information about $\overline{\rho}_{E,p}$ when the prime $q$ in Theorem \ref{mult} splits in $K'$, in order to prove Theorem \ref{thm4}. We then combine this information with Theorem \ref{twistthm} to prove irreducibility for most values of $p$ and $d$ in Table 1.

Given $N>0$, write $g_1, \dots, g_k$ for representatives of the Galois conjugacy classes of Hecke eigenforms in the space of cuspforms $S_2(N)$. To each $g_i$ is associated a simple abelian variety $A_i / \Q$ and we describe the Galois conjugates of $g_i$ as the \emph{cuspforms attached to} $A_i$. Each $g_i$ arises from a newform at some level $M_i \mid N$. Writing $m_i$ for the number of divisors of $N/M_i$, we have  \[ J_0(N) \sim A_1^{m_1} \times \dots \times A_k^{m_k}, \] where $\sim$ denotes isogeny over $\Q$ (see \citep[p.~3481]{murty}, for example). We see that $\mathrm{Rk}(J_0(N)) = \sum_{i=1}^k m_i \cdot  \mathrm{Rk}(A_i).$ 

We are interested in the ranks of the isogeny factors $A_i$. Write $L(A_i,s)$ for the $L$-function of $A_i$. A theorem of Kolyvagin and Logachev \citep{klthm} asserts that if $L(A_i,1) \neq 0$, then $\mathrm{Rk}(A_i(\Q))=0$, and using \texttt{Magma} we can verify whether or not $L(A_i,1)$ is zero. We define $\mathcal{A}_0 / \Q$ as the largest rank $0$ quotient of $J_0(N)$.

When $N=p$ is prime, we can write \[J_0(p) \sim J_0^+(p) \times J_0^-(p), \] where $J_0^+(p)$ is the Jacobian of the modular curve $X_0^+(p)$, and $J_0^-(p) = J_0(p) / (1+w_p)$. When $p>7$, we denote by $J_e(p)$, or simply $J_e$, the \emph{Eisenstein quotient} of $J_0(p)$, as constructed by Mazur in \citep{eisenstein}. This is a non-trivial factor of $J_0^-(p)$ satisfying $\mathrm{Rk}(J_e(\Q)) = 0$. In fact, $J_e(\Q)$ is cyclic of order $n$, where $n$ is the numerator of $(p-1)/12$.

Before proving Theorem \ref{mult}, we first prove the following lemma.

\begin{lemma}[Formal immersion criterion]\label{FIM} Let $N=p$ or $2p$ such that $g(X_0(N)) \ge 2$. Denote the cusps of $X_0(N)$ by $\infty = c_1,c_2 \dots, c_m$. Let $f_1, \dots, f_t$ be cuspforms attached to $\mathcal{A}_0$. Let $(y,z) \in X^{(2)}_0(N)(\Q)$. Let $q \ne 2,p$ be a rational prime such that $(y,z)_{\mathbb{F}_q} = (c_i,c_j)_{\mathbb{F}_q}$ for some $1 \leq i,j \leq m$. Denote by $a_n(f,c_k)$ the $n$th coefficient of $f$ expanded at the cusp $c_k$, and define matrices \[ F_{\infty}:=\begin{pmatrix}
a_1(f_1,\infty) & a_2(f_1,\infty) \\ \vdots & \vdots \\ a_1(f_t,\infty) & a_2(f_t,\infty)
\end{pmatrix}  ~~ \text{and} ~~ 
F_{\infty,k}:=\begin{pmatrix}
a_1(f_1,\infty) & a_1(f_1,c_k) \\ \vdots & \vdots \\ a_1(f_t,c_1) & a_1(f_t,c_k)
\end{pmatrix} \] for $2 \leq k \leq m$. If $F_\infty$ and $F_{\infty,k}$ all have rank $2$ modulo $q$ then $(y,z) = (c_i,c_j)$.
\end{lemma}

\begin{proof}
Since $N=p$ or $2p$, the cusps of $X_0(N)$ are rational, and the set of Atkin--Lehner involutions acts transitively on the cusps. So we may assume that  $(y,z)_{\mathbb{F}_q} =(\infty,c_k)_{\mathbb{F}_q}$, for some $1 \leq k \leq m$. Conisder the following Abel--Jacobi map: \begin{align*} \iota_k:  X_0^{(2)}(p) & \longrightarrow J_0(p) \\ (u,v) & \longmapsto [u+v-(\infty+c_k)]. 
\end{align*} Let $h:X_0^{(2)}(N) \rightarrow \mathcal{A}_0$ denote the composition of $\iota_k$ with the projection map $J_0(p) \rightarrow \mathcal{A}_0$.
Each cuspform $f_i$ gives rise to an element $\omega_i \in \mathrm{Cot}(\mathcal{A}_0) \hookrightarrow \mathrm{Cot}(J_0(N))$. We then apply the formal immersion criterion as stated in \citep[p.~16]{sporadic} to conclude that $h$ is a formal immersion at $(\infty,c_k)_{\mathbb{F}_q}$, and since $\mathrm{Rk}(\mathcal{A}_0) = 0$, we conclude that $(y,z) = (\infty,c_k)$.
\end{proof}

\begin{proof}[Proof of Theorem \ref{mult}] Suppose $p>19$, so that $g(X_0(p)) \ge 2$. Since $\overline{\rho}_{E',p}$ is reducible, $E'$ gives rise to a non-cuspidal $K'$-point, which we denote $x$, on the modular curve $X_0(p)$. So the pair $(x,x^\sigma)$, is a rational point on $X_0^{(2)}(p)$. Since there is a unique prime of $K'$ above $q$, which is of multiplicative reduction for $E'$, it follows that $(x,x^\sigma)_{\mathbb{F}_q} = (\infty,\infty)_{\mathbb{F}_q}$ or $(0,0)_{\mathbb{F}_q}$. After applying the Atkin--Lehner involution $w_p$ to the pair $(x,x^\sigma)$ if necessary, we may assume that $(x,x^\sigma)_{\mathbb{F}_q} = (\infty,\infty)_{\mathbb{F}_q}$.

We now split into two cases. First, suppose $X_0(p)$ is non-hyperelliptic and denote by $\iota$ the Abel--Jacobi map on $X_0^{(2)}(p)$ with base point $(\infty,\infty)$. We mainly follow the proof of \citep[pp.~4-6]{najturc} which is based on the work of Kamienny in \citep[pp.~223-225]{kamienny}. We show that the matrix $F_\infty$ has rank $2$ modulo $q$, as this implies, by Lemma \ref{FIM}, that $(x,x^\sigma) = (\infty,\infty)$, a contradiction, since $x$ is a non-cuspidal point.
It therefore suffices to find a pair of newforms $f_1 = \sum a_i q^i$ and $f_2 = \sum b_i q^i$ (here $a_1=b_1=1$) attached to $J_e$ such that $a_2 \not\equiv b_2 \pmod{q}$. Equivalently, it is enough to show that the Hecke operator $T_2$ does not act as a scalar on $J_e$ mod $q$. For $p>61$, this is proven in \citep[pp.~224-225]{kamienny} using the properties of $J_e$. For $p=43,53,$ and $61$ (the remaining primes $>19$ with $X_0(p)$ non-hyperelliptic), the characteristic polynomial of $T_2$ on each of the Eisenstein quotients $J_e(p)$ is displayed in \citep[p.~226]{kamienny}. We see that $T_2$ does not act as a scalar modulo any prime $>5$, and so we also have a formal immersion at $(\infty,\infty)_{\mathbb{F}_q}$ in these cases too.

For the second case, we suppose that $p \ne 37$, and that $X_0(p)$ is hyperelliptic. We follow the argument of \citep{kami2}. For these values of $p$, the Jacobian $J_0(p)$ has rank $0$ over $\Q$, and the hyperelliptic involution on $X_0(p)$ is the Atkin--Lehner involution $w_p$ (see Theorem \ref{oggbars}). Since $J_0(p)$ is finite, arguing similarly to above, $\iota(x,x^\sigma) = 0$ in $J_0(p)(\Q)$, so there exists a degree $2$ rational function $g$ on $X_0(p)$ satisfying $\mathrm{div}(g) = x+ x^\sigma - 2 \infty$. As $g$ has degree $2$ and $X_0(p)$ is hyperelliptic, the hyperelliptic involution, $w_p$, must fix $\infty$, a contradiction, since $w_p$ interchanges the two cusps of $X_0(p)$.
\end{proof}

Unfortunately, it does not seem possible to explicitly construct inert primes of multiplicative reduction for the Frey curve $E_{a,b,c,p}$, and so in our situation, this result seems limited to its use in the proof of Proposition \ref{goodred}. We can, however, using Lemma \ref{krauslem} (usually) find split primes $q$, for which both primes of $K$ above $q$ are of multiplicative reduction for $E$. Although results as strong as Theorem \ref{mult} are not possible in this case, we can still extract useful information. The key difference for a split prime is that if we try to follow the argument used in the proof above, $(x,x^\sigma)$ may reduce to $(\infty,0)_{\mathbb{F}_q}$.

\begin{theorem}\label{splitprime} Let $E'$ be an elliptic curve defined over any quadratic field $K'$.  Let $p$ be a prime, and suppose $\overline{\rho}_{E',p}$ is reducible, so that $E'$ gives rise to a point $x \in X_0(p)(K')$. Let $q>5$, $q \neq p$ be a rational prime such that $q$ splits in $K'$, and both primes of $K'$ above $q$ are of multiplicative reduction for $E'$. Then \begin{enumerate}[(i)]
\item either $p \leq 19$ or $p=37$; 
\item or $w_p(x) = x^\sigma$ on $X_0(p)$.
\end{enumerate}
\end{theorem}

The main ingredient in this proof is to use Siksek's relative symmetric Chabauty criterion  \citep[223-224]{Symmchab}.

\begin{lemma}[Relative symmetric Chabauty criterion]\label{chabcrit} Let $p>19$ and let $q \ne 2$, $p$ be a rational prime. Suppose $(y,z) \in  X_0^{(2)}(p)(\Q)$ satisfies $(y,z)_{\mathbb{F}_q} = (\infty,0)_{\mathbb{F}_q}$. Then $w_p(y) = z$.
\end{lemma}

\begin{proof}
Denote by $\psi$ the degree $2$ map $X_0(p) \rightarrow X_0^+(p)$. Since $q \ne p$, both $X_0(p)$ and $X_0^+(p)$ have good reduction at $q$, and since $p>19$, we have $g(X_0(p)) \geq 2$. The argument we use here has links to Mazur's argument in \citep[pp.~142-143]{ratisog}.

 We write $\underline{q} = e^{2\pi i \tau}$, so as to differentiate it from the prime $q$. Let $f_e = \sum_{i \geq 1} a_i \underline{q}^i$, with $a_1=1$, be a newform attached to the Eisenstein quotient, $J_e$, of $J_0(p)$, and write $\omega_e = \sum_{i \geq 1} a_i \underline{q}^{i-1} d\underline{q} \in \mathrm{Cot}(J_e)$ for the corresponding differential, which we view as a global $1$-form on $X_0(p)$ via the inclusion $ \mathrm{Cot}(J_e) \hookrightarrow H^0(X_0(p), \Omega^1)$.

We now apply the relative symmetric Chabauty criterion, as stated in \citep[223-224]{Symmchab}, to the point $(\infty,0)$. We also use the argument of \citep[p.~10]{mquad} to allow $q=3$. Using the terminology of \citep{Symmchab}, $\underline{q}$ acts as a well-behaved uniformiser at $\infty$. The differential $\omega_e$ satisfies three important properties which allows us to apply the criterion. \begin{itemize}
\item The differential $\omega_e$ is \emph{annihilating}. Since $J_e(\Q)$ has rank $0$, $\omega_e$ lies in the kernel on the left of the integration pairing described in \citep[p.~214]{Symmchab}.
\item The differential $\omega_e$ has \emph{zero trace}. The Atkin--Lehner involution $w_p$ induces the trace map (with respect to $\psi$) on global $1$-forms \[1+w_p^*:  H^0(X_0(p), \Omega^1) \longrightarrow H^0(X_0^+(p), \Omega^1). \] Since the projection map $J_0(p) \rightarrow J_e$ factors via $J_0^-(p) = J_0(p)/(1+w_p)$, the $1$-form $\omega_e$ lies in the kernel of this trace map.
\item The differential $\omega_e$ has $\underline{q}$-expansion $(a_1+a_2\underline{q}+a_3\underline{q}^2 + \cdots ) d\underline{q}$, with $a_1 = 1 \not\equiv 0 \pmod{q}$.
\end{itemize} 
Applying the criterion, it follows that $(y,z) \in \psi^*(X_0^+(p)(\Q))$. So $w_p(y)=z$.
\end{proof}

\begin{remark} For $X_0(p)$ hyperelliptic with $p \ne 37$, we could instead prove Lemma \ref{chabcrit} by repeating the argument of the proof of the hyperelliptic case of Theorem \ref{mult}, replacing $(\infty,\infty)$ by $(\infty,0)$.
\end{remark}

The proof of Theorem \ref{splitprime} is then a straightforward consequence of this result.

\begin{proof}[Proof of Theorem \ref{splitprime}]
Since both primes of $K'$ above $q$ are of multiplicative reduction for $E'$, after applying $w_p$ if necessary, $(x,x^\sigma)_{\mathbb{F}_q} = (\infty,\infty)_{\mathbb{F}_q}$ or $(\infty,0)_{\mathbb{F}_q}$. If $(x,x^\sigma)_{\mathbb{F}_q} = (\infty,\infty)_{\mathbb{F}_q}$, we are in the situation of the proof of Theorem \ref{mult} and we conclude that $p \leq 19$ or $p =37$; so we instead suppose that $p>19$, and that $(x,x^\sigma)_{\mathbb{F}_q} = (\infty,0)_{\mathbb{F}_q}$. We then apply Lemma \ref{chabcrit} to conclude that $w_p(x) = x^\sigma$. \end{proof}

We note that Theorems \ref{mult} and \ref{splitprime} combine to give Theorem \ref{thm4}, which is stated in the introduction.

\begin{example} We apply Theorems \ref{splitprime} and \ref{twistthm} to show that $\overline{\rho}_{E,103}$ is irreducible in the case $d=26$. Using Lemma \ref{krauslem}, we find that both primes of $\Q(\sqrt{26})$ above $1031$ are of multiplicative reduction for $E$. Applying Theorem \ref{splitprime} we deduce that $E$ gives rise to a point $x \in X_0(103)(\Q(\sqrt{26}))$ satisfying $w_p(x) = x^\sigma$. So $E$ gives rise to a rational point on the twisted curve $X_0^{26}(103)$. However, applying Theorem \ref{twistthm} with the prime $l = 13$, we find that $X_0^{26}(103)(\Q_{13}) = \emptyset$, a contradiction.
\end{example}

As mentioned in Section 4.1, the case $p=37$ is rather special, since the modular curve $X_0(37)$ is both hyperelliptic and bielliptic, with an elliptic quotient of positive rank. Moreover, the hyperelliptic involution on $X_0(37)$ is not of Atkin--Lehner type.  To prove irreducibility in the case $p=37$, we choose instead to use Lemma \ref{FIM} and work on the curve $X_0(74)$. Our strategy is similar to that of \citep[pp.~19-21]{sporadic}.

\begin{lemma} Let $E'$ be an elliptic curve defined over any quadratic field $K'$, with a $2$-torsion point defined over $K'$. Suppose there exists a prime $q \ne 2,37$ such that $E'$ has multiplicative reduction at all primes of $K'$ above $q$. Then $\overline{\rho}_{E',37}$ is irreducible.
\end{lemma}

\begin{proof} Since $E'$ has a $2$-torsion point defined over $K'$, it gives rise to a point $x \in X_0(74)(K')$, so that $(x,x^\sigma) \in X_0^{(2)}(74)(\Q)$. 
The Jacobian of $X_0(74)$ decomposes as the following product of abelian varieties: \[ J_0(74) \sim E_a \times E_a \times A_1 \times A_2 \times E_b \times E_b. \] Here, $E_a$ is the elliptic curve `37a1' of rank $1$, $E_b$ is the elliptic curve `37b1' of rank $0$, and $A_1$,$A_2$ are $2$-dimensional abelian varieties of rank $0$. We have $\mathcal{A}_0 = A_1 \times A_2 \times E_b \times E_b$. Let $f_1, \dots, f_6$ be cuspforms attached to $\mathcal{A}_0$ (here, $f_5=f_6$, so we can exclude $f_6$ if we like). Using Lemma \ref{FIM}, it suffices to check that the matrices $F_\infty, F_{\infty,2}, F_{\infty,3},$ and $F_{\infty,4}$ have rank $2$ modulo $q$ (for any $q \ne 2,37$). We verified that this is the case using \texttt{Magma}. Note that to expand some $f_i$ at a cusp other than $\infty$, we can apply the appropriate Atkin--Lehner involution.
\end{proof}

Using the results of this section in combination with Lemma \ref{krauslem} and Theorem \ref{twistthm} proves irreducibility for most cases appearing in Table 1. The remaining cases are displayed in Table 2. The techniques we explore in the remainder of this section will eliminate these cases, as well as provide alternative strategies for dealing with many values of $p$ and $d$ appearing in Table 1.

\begingroup
\begin{table}[ht!]\label{Tab4}
\begin{center}
\footnotesize
\begin{tabular}{ |c||c|c|c|c|c|}
\hline 
$d$ & $29$  & $53$ & $61$ & $74$ & $89$ \\
\hline
$p$ & $29$  & $53$ & $61$ & $43$ & $53$ \\
\hline
\end{tabular}
\caption{Irreducibility Step 2}
\end{center}
\end{table}
\endgroup
\normalsize

\subsection{Preimages under Modular Parametrisation}

Let $E'$ be an elliptic curve defined over $\Q$ of conductor $N$. Then $E'$ admits a map defined over $\Q$, called the \emph{modular parametrisation} of $E'$: \[ \varphi: X_0(N) \rightarrow E'. \]  We will assume $E'$ is \emph{optimal} so that $\varphi$ is unique up to sign and maps the cusp at infinity on $X_0(N)$ to the identity of $E'$. Write $m$ for the degree of the modular parametrisation, which we refer to as the \emph{modular degree} of $E'$. We note here that the curve $E'$ and the Frey curve $E$ are not related. Their conductors $N$ and $\mathcal{N}$ are also not related. For background on the modular parametrisation map we refer the reader to \citep{firstcourse, yang, hao}.

Using the map $\varphi$ to understand quadratic points over a fixed quadratic field is based on the following observation.

\begin{lemma}Let $E'$ be an optimal elliptic curve of conductor $N$ and write $\varphi: X_0(N) \rightarrow E'$ for its modular parametrisation. Suppose $E'(\Q) = E'(L)$ for $L$ a number field. Then $X_0(N)(L) \subseteq \varphi^{-1}(E'(\Q))$.
\end{lemma}

We would like to compute the fields of definition of preimages of points under the modular parametrisation map. For our purposes, this is only useful when $E'(\Q)=E'(K)$ is finite, but the techniques we develop apply even if this is not the case. 

Suppose $E'$ is given by a Weierstrass equation \[ y^2 + a_1xy+a_3y=x^3+a_2x^2+a_4x+a_6, \] with $a_i \in \Q$. Then $x$ and $y$ are rational functions on $E'$ of degree $2$ and $3$ respectively. We can then pull these back via $\varphi$ to obtain rational functions on $X_0(N) = \Gamma_0(N) \backslash \mathcal{H}^*$:  \[ x(q) = \varphi^*(x) \quad \text{and} \quad y(q) = \varphi^*(y), \] where $q=e^{2 \pi i z}$ for  $z \in \C$.

Using \texttt{SageMath} we can compute the $q$ expansions of $x(q)$ and $y(q)$. The rational functions $x(q)$ and $y(q)$ satisfy the equation of the elliptic curve $E'$, as well as the relation \[ \frac{dx(q)}{2y(q)+a_1x(q)+a_3} = \frac{f(q)dq}{q}, \] where $f$ is the rational newform of level $N$ corresponding to the isogeny class of $E'$. The rational functions $x(q)$ and $y(q)$ on $X_0(N)$ have degrees $2m$ and $3m$ respectively.

We would like to obtain a planar model for the curve $X_0(N)$ by finding a relation between $x(q)$ (or $y(q)$) and another rational function on $X_0(N)$. Such a relation always exists.

\begin{lemma}[{\citep[p.~24]{hao}}] Let $r(q),s(q) \in \Q(X_0(N))$ be rational functions on $X_0(N)$. Then there exists an irreducible polynomial $F \in \Q[R,S]$, which we call a \emph{minimal polynomial relation}, such that $F(r(q),s(q))=0$, and $\deg_R(F) \leq \deg(s(q))$ and  $\deg_S(F) \leq \deg(r(q))$.   
\end{lemma}

Moreover, the following result tells us that we need only check a polynomial relation up to certain precision.

\begin{lemma}[{\citep[p.~28]{hao}}]
Let $r(q),s(q) \in \Q(X_0(N))$ be rational functions on $X_0(N)$.  Suppose $G \in \Q[R,S]$ satisfies $G(r(q),s(q))=O(q^M)$ for some integer $M > 2\deg(r)\deg(s)$. Then $G(r(q),s(q))=0$.
\end{lemma}

Let $r(q),s(q) \in \Q(X_0(N))$ and let $F \in \Q[R,S]$ be a minimal polynomial relation for these rational functions. Then usually $F$ will have degree $\deg(s)$ in $R$ and degree $\deg(r)$ in S. If this is the case then $F$ gives a planar model for the modular curve $X_0(N)$. If the degrees are less than these maxima, then we will obtain a model for a quotient of $X_0(N)$. For example, the equation of the elliptic curve $E'$ is  a minimal polynomial relation between $x(q)$ and $y(q)$.

We aim to find a minimal polynomial relation involving between $x(q)$ and another rational function on $X_0(N)$ that will give us a planar model for our modular curve. A natural first choice is the $j$-function: \[ j(q) = q^{-1} + 744 + 196884q + 21493760q^2  + \mathcal{O}(q^3)  \in \Q(X_0(N)) \text{ for all } N. \] The degree of $j(q)$ is given by the index of $\Gamma_0(N)$ in the full modular group $\mathrm{SL}_2(\mathbb{Z})$. As $N$ gets large, say $N \geq 50$, it quickly becomes impractical to compute a minimal polynomial relation between $x(q)$ and $j(q)$, and so we seek to replace $j(q)$ with a rational function on $X_0(N)$ of smaller degree. We do this using eta products.

First define \emph{Dirichlet's eta function} as \[\eta(q) := \frac{1}{q^{24}}\prod_{n=1}^{\infty}(1-q^n).\]
An \emph{eta product of level $N$} (also referred to in the literature as an eta quotient) is then given by
\[ s(q) = \prod_{d \mid N} \eta(q^d)^{r_d} = \prod_{d \mid N} \eta_d^{r_d}, \] for some integers $r_d$, and where we write $\eta_d$ for $\eta(q^d)$. Such an eta product need not be a rational function on $X_0(N)$, but if the integers $r_d$ satisfy certain conditions, then it is. 

\begin{theorem}[Ligozat's Criteria {\citep[p.~28]{lig}}] An eta product of level $N$, $\prod_{d \mid N} \eta_d^{r_d}$, is a rational function on $X_0(N)$ if the following conditions are satisfied:
\begin{align*}
& 1.~~ \sum_d r_d \frac{N}{d} \equiv 0 \pmod{24};  \qquad 
2. ~~ \sum_d r_d d \equiv 0 \pmod{24};   \\
& 3. ~~ \sum_d r_d = 0;   \qquad  \qquad \qquad \qquad 
4. ~~ \sum_{d \mid N} \left(\frac{N}{d}\right)^{r_d} \in \Q^2.
\end{align*}

\end{theorem}

We note here that the support of the divisor of an eta product $s(q) \in \Q(X_0(N))$ is contained in the set of cusps of $X_0(N)$. Using \texttt{SageMath} we can find a basis for the group of eta products of level $N$ that are rational functions on $X_0(N)$ (i.e. that satisfy Ligozat's criteria). We can then choose any one of these (or some combination), say $s(q)$. It is natural to start by choosing a basis element of minimal degree. We find a minimal polynomial relation between $x(q)$ and $s(q)$, and substituting in $x$-coordinates of points in $E'(\Q)$ will give the $s$-values of the preimages of points under the modular parametrisation map, and from this we can often deduce their field of definition.

There are two issues that can arise here. The first is that the minimal polynomial relation may not be of maximal degree in its two variables. This will occur when the map $x(q) \times s(q): X_0(N) \rightarrow \mathbb{P}^1 \times \mathbb{P}^1$ (viewing $x(q)$ and $s(q)$ as morphisms from $X_0(N)$ to $\mathbb{P}^1$) is not injective, and so the equation we obtain gives a planar model for a curve $Y$ which is a finite quotient of $X_0(N)$. It is still possible to recover some information in this case.

The other issue is that the $s$-values we obtain may not give the field of definition of the points in the preimage. For example, if a degree $2$ rational divisor on $X_0(N)$ has an $s$-value of $a \in \Q$, with multiplicity $2$, then this could be due to two rational points, or a pair of quadratic points. By considering the cusps it is most likely we can conclude it is a pair of quadratic points, but we still do not know their field of definition, which is what we are ultimately interested in. We will see this in the example below.

To overcome both of these problems, we can usually simply replace our eta product by a different one, and if necessary combine information from multiple eta products as we see below.

\begin{example}[Eta Product Method for $X_0(116)$]

We consider here the case $K= \Q(\sqrt{29})$ and $p=29$. There are no elliptic curves of rank $0$ over $K$ with conductor $29$ or $58$, but the elliptic curve $E'$ with Cremona label `116b1' has Mordell--Weil group $\Z / 3\Z$ over both $\Q$ and $K$. The curve is given by \[E': y^2 = x^3+x^2-4x+4. \] We have \[ E'(\Q) = E'(\Q(\sqrt{29})) = \{0_{E'},R,-R \}, \] where $R$ has $x$-coordinate $0$. The modular degree of $E'$ is $8$.

We work with the modular curve $X_0(116)$ which is of genus $13$ and has six rational points: the six cusps. We denote these cusps by $c_\infty, c_0, c_2, c_4, c_{29},$ and $c_{58}$. We would like to show in this case that $X_0(116)(\Q(\sqrt{29})) = X_0(116)(\Q)$ as this will prove that $\overline{\rho}_{E,29}$ is irreducible.

We find a basis, using \texttt{SageMath}, for the group of eta products at level $116$. This basis has five elements. The first four have degree $12$ as rational functions on $X_0(116)$, and the fifth has degree $14$. We start by choosing the first basis element and find the minimal polynomial relation $F_1(X,S)$ between $x(q)$ and $s_1(q)$. This polynomial has degree $6$ in $X$ and $8$ in $S$, and so does not give a planar model for $X_0(116)$, but rather a degree $2$ quotient of $X_0(116)$. Although we can still obtain information from this, we instead work with the second basis element \[s_2 = \eta_1^{-3} \cdot  \eta_2^4 \cdot \eta_4^{-1} \cdot  \eta_{29}^{-1} \cdot  \eta_{58}^4 \cdot \eta_{116}^{-3}, \] with divisor \[ (s_2) = -6(c_\infty) -6(c_0)+ 5(c_2) + (c_{4}) + (c_{29}) + 5(c_{58}) . \] We calculate a minimal polynomial relation $F_2(X,S)$ for $x(q)$ and $s_2(q)$ and find that this time it has degree $12$ in $X$ and degree $16=2m$ in $S$, so we have obtained a planar model for our curve. Some terms of this polynomial are as follows \[ F_2(X,S) = X^{12}S^{14} + \cdots + X^6S^{16} + \cdots + 1048576XS^2 - 4194304S^3. \]
Substituting in the value $0 = x(R) = x(-R)$ we find that  \begin{align*}F_2(0,S) = & -4096 (S)^3(S-2)^2(S^2-8S+8)^2 (S^2+2S+2) \\ & (S^4-4S^3+6S^2-4S+2).
\end{align*}
The factor $S^3$ corresponds to three cusps in the preimage, $\varphi^{-1}(\{R,-R\})$, of $R$ and $-R$. The factor  $S^4-4S^3+6S^2-4S+2$ corresponds to a tuple of quartic points defined over the cyclotomic extension $\Q(\zeta_8)$, and the factor $S^2+2S+2$ corresponds to a pair of quadratic points defined over $\Q(\sqrt{-1})$. The factor $(S^2-8S+8)^2$ could either correspond to a pair of quadratic points defined over $\Q(\sqrt{2})$, or a tuple of quartic points. The factor $(S-2)^2$ may correspond to a pair of rational points or a quadratic point. We know $ |X_0(116)(\Q) | = 6$ and so we will shortly be able to see that these rational $s_2$-values do not arise from a pair of rational points, so $(S-2)^2$ corresponds to a pair of quadratic points, but we cannot say over which quadratic field they are defined. Finally, this factorisation does not display any poles of $s_2$ appearing in the preimage of $R$ or $-R$. In the factorisation of $T^{16}F_2(0,1/T)$ (which gives the $1/s_2$-values in the preimage) there is a factor $T$. This corresponds to a pole of $s_2$, and shows that there is a fourth cusp in preimage.

In order to understand the preimage of $0_{E'}$ we first define $G_2(Z,S):=Z^{12}F_2(1/Z,S)$, as setting $Z=0$ will correspond to setting $ X = \infty$. We find \[ G_2(0,S) = S^2(S^2+2S+2)^2(S^4-4S^3+6S^2-4S+2)^2. \] We note $G_2(0,S)$ is a square since $-0_{E'} = 0_{E'}$. By considering this factorisation in conjunction with $T^{16}G(0,1/T)$ we see that we have two cusps, a pair of quadratic points defined over $\Q(\sqrt{-1})$, and a tuple of quartic points defined over $\Q(\zeta_8)$.

In order to understand more about the fields of definition of the preimages of $R$ and $-R$, we use the fifth basis element \[ s_5 = \eta_1^2 \cdot \eta_2^{-2} \cdot \eta_4^2 \cdot \eta_{29}^{-2} \cdot \eta_{58}^2 \cdot \eta_{116}^{-2}, \] with divisor \[ (s_5) = -7(c_\infty) +7(c_0)- 7(c_{4}) + 7(c_{29}). \] Note that this divisor is only supported on four of the six cusps. This eta product has degree $14$. We find a minimal polynomial relation $F_5(X,S)$. It has degree $14$ in $X$ and $16$ in $S$, and we have \begin{align*} F_5(0,S) = & 16384 (S)(S+1)(S+29)(S^2-10S+29)^2(S^2+4S+29) \\  & (S^2+10S+29) (S^4-28S^3+272S^2-812S+841).\end{align*} As before, we recover four cusps, and a tuple of quartic points. The quadratic factors $(S^2+4S+29)$ and $(S^2+10S+29)$ both correspond to pairs of quadratic points defined over $\Q(\sqrt{-1})$. We could see the field of definition of one of these pairs using $s_2$, but the other pair had $s_2$-values $2$ and we could not deduce its field of definition. We now see that this quadratic point is defined over $\Q(\sqrt{-1})$. We also note that the factor $(S^2-10S+29)^2$ corresponds to either a pair of quadratic points defined over $\Q(\sqrt{-1})$, or a tuple of quartic points. From its $s_2$-value, we know it must be either a  pair of quadratic points defined over $\Q(\sqrt{2})$, or a tuple of quartic points. Combining these two pieces of information, we conclude it must be a tuple of quartic points, and that its field of definition is a quadratic extension of both $\Q(\sqrt{-1})$ and $\Q(\sqrt{-2})$, so the points must be defined over $\Q(\sqrt{-1}, \sqrt{2}) = \Q(\zeta_8)$.

We conclude that $\varphi^{-1}(E'(\Q)) = \varphi^{-1}(E'(\Q(\sqrt{29})))$ is made up of six rational cusps, three pairs of quadratic points defined over $\Q(\sqrt{-1})$, and three tuples of quartic points defined over $\Q(\zeta_8)$, making up $24 = 8 \cdot | E'(K) |$ points in total. This proves that  $X_0(116)(\Q(\sqrt{29})) = X_0(116)(\Q)$, and in fact determines $X_0(116)(L)$ for any number field $L$ with $E'(\Q) = E'(L)$.
\end{example}

We applied the same techniques to prove irreducibility for the remaining values of $d$ and $p$ appearing in Table 2, other than $d=61,p=61$ and $d=74,p=43$, which we consider in Section 4.4.

\subsection{A Mordell--Weil Sieve}
In this section, we study the modular curves $X_0(43)$ and $X_0(61)$ to prove irreducibility in the cases $d=61,p=61$ and $d=74,p=43$, thus completing the proof of Theorem \ref{irred}. The curves $X_0(61)$ and $X_0(43)$ are bielliptic and non-hyperelliptic. Their quotients $X_0^+(43)$ and $X_0^+(61)$ are the elliptic curves with Cremona labels `43a1' and `61a1' respectively, each of which has rank $1$ and trivial torsion. We employ a version of the Mordell--Weil sieve to study $X_0^d(N)(\Q)$. We illustrate the sieving method for $X_0(61)$, although the same techniques will apply for other curves, and this method also has some overlap with the other methods we have seen. For a general introduction to the (usual) Mordell--Weil sieve, we refer the reader to \citep{exposieve}.

The curve $X_0(61)$ has genus $4$. Using the `small modular curves' package in \texttt{Magma}, we obtain a smooth model for this curve, the Atkin--Lehner involution $w_{61}$, and the $j$-map.
We start by obtaining a model for which the Atkin--Lehner involution is diagonalised. We do this by finding a matrix diagonalising $w_{61}$ and applying the corresponding coordinate change to the equations of our curve. We obtain the following model in $\mathbb{P}^3$:

\begin{align*} & -4Y^2-4XZ+Z^2  =  T^2,  \\
& X^3 - X^2Y - 3XY^2 - X^2Z + XYZ + Y^2Z - YZ^2  = 0.
\end{align*}

We see that this is the intersection of a quadric and a cubic surface. We write $F(X,Y,Z) = -4Y^2-4XZ+Z^2$, and $G(X,Y,Z)$ for the homogeneous cubic in the second defining equation of the above model. The coordinate change we have applied introduces $2$ as a prime of bad reduction for this model. The Atkin--Lehner involution is now given by \begin{align*}
w_{61}: X_0(61) & \longrightarrow X_0(61)  \\ 
(x: y: z: t) &  \longmapsto (x: y: z: -t). 
\end{align*}
We see from these equations that we have the degree $2$ map \begin{align*} \psi: X_0(61) & \longrightarrow X_0^+(61) \\
(x:y:z:t) &  \longmapsto (x:y:z), 
\end{align*} with $X_0^+(61)$ the elliptic curve defined by $G(X,Y,Z) = 0$ in $\mathbb{P}^2$. 

We suppose, hoping to obtain a contradiction, that our Frey curve $E$ gives rise to a non-exceptional quadratic point on $X_0(61)(\Q(\sqrt{d}))$, which we denote by $P$ here (instead of $x$). We are interested in the case $d=61$, but we in fact obtained a contradiction for all $d>0$ we tested. As $P$ is a non-exceptional point, $w_{61}(P) = P^\sigma$. It follows that $P$ can be expressed as $P = (x: y: z: b\sqrt{d})$ with $x,y,z,b \in \mathbb{Q}$. As $\psi(P) \in X_0^+(61)(\Q)$, we have that $\psi(P) = m \cdot R$, for some $m \in \Z$, where $R$ generates the group $X_0^+(61)(\Q) \cong \Z$. 

Choose a prime $l \nmid d$  of good reduction for both $X_0(61)$ and $X_0^+(61)$ (given by the above models); in particular, $l \nmid 2 \cdot 61 \cdot d$. Write $N_l$ for the order of $R$ in the reduction of $X_0^+(61)$ modulo $l$. Write $k$ for the residue field of $K$ modulo a prime above $l$. This will either be $\mathbb{F}_l$ or $\mathbb{F}_{l^2}$. We have the following commutative diagram, where $\sim$ denotes reduction modulo $l$:

\begin{center}
\begin{tikzcd} X_0(61)\arrow[r, "\psi"] \arrow[d, "\sim"] & X_0^+(61) \arrow[d, "\sim" ] \\  \widetilde{X}_0(61)  \arrow[r,  "\widetilde{\psi}"]& \widetilde{X}_0^+(61)  \end{tikzcd}
\end{center}

Since $\psi(P) = m \cdot R$, we see that $\widetilde{\psi}(\widetilde{P}) = \overline{m} \cdot \widetilde{R}$, where $\overline{m} \equiv m \pmod{N_l}$. So $\widetilde{P} \in \widetilde{\psi}^{-1}(\overline{m} \cdot \widetilde{R})$. Fix $m_0 \in \{0, \cdots, N_l-1 \}$. Then we can explicitly compute the set $\widetilde{\psi}^{-1}(m_0 \cdot \widetilde{R})=\{Q_1,Q_2 \}$, where $Q_1 = (u,v,w,s)$ and $Q_2 = \widetilde{w_{61}}(Q_1) = (u,v,w,-s)$, with $u,v,w \in \mathbb{F}_l \subseteq k$ and $s \in k$. We note that we may have $Q_1 = Q_2$.

We would like to try and argue that $\widetilde{P} \notin \{Q_1,Q_2 \}$ if possible, as we can then conclude that $m \not\equiv m_0 \pmod{N_l}$. There are two strategies we can use.

\begin{enumerate}
\item The point $P = (x : y : z : b\sqrt{d})$ satisfies $F(x,y,z) = db^2$, and so reducing mod $l$ we have \[ F(\widetilde{x},\widetilde{y},\widetilde{z}) \cdot \widetilde{d}^{-1} \equiv \widetilde{b}^2 \pmod{l}.\] It follows that $F(u,v,w) \cdot \widetilde{d}^{-1}$ is a square mod $l$, so if this is not the case, then $m \not\equiv m_0 \pmod{N_l}$.

\item The Frey curve $E$ has full two-torsion over $K$. If $\widetilde{P}$ is not a cusp on the reduced modular curve $\widetilde{X}_0(61)$, then it corresponds to an elliptic curve with full two-torsion over $k$. Also, since $w_{61}(P)$ corresponds to the elliptic curve $E/C$ where $C$ is a cyclic subgroup of order $61$, it follows that $w_{61}(P)$ also has full two-torsion over $K$. Suppose $l>3$ (to avoid the $j$-invariant in characteristic $3$). Then if $Q_1$ is not a cusp, and all elliptic curves over $k$ with $j$-invariant $\widetilde{j}(Q_1)$ (this consists of two elliptic curves when $\widetilde{j}(Q_1) \not\equiv 0, 1728 \pmod{l}$, and a maximum of six elliptic curves otherwise) do not have full-two torsion over $k$, then we have a contradiction, and we conclude $m \not\equiv m_0 \pmod{N_l}$. 
\end{enumerate}

We combine these two methods of elimination to obtain a list of possibilities for $m \pmod{N_l}$. We then repeat this process with a list of primes $\{l_1, \cdots, l_r\}$, and use the Chinese remainder theorem to obtain a list of possibilities for $m \pmod{N}$, where $N := \gcd(N_{l_1}, \cdots, N_{l_r})$. If this list of possibilities is empty then we obtain our desired contradiction. We choose our primes $l_i$ so that the orders  $N_{l_i}$ are small and share many prime factors. This helps avoid a combinatorial explosion due to the Chinese remainder theorem, and also increases the likelihood of obtaining contradictory information. For the curve $X_0(61)$ and $d=61$, using the primes $5,7,11,13$ sufficed to reach a contradiction.

 We note the importance of using both elimination steps in the sieve. If we do not sieve using $j$-invariants, then we found that the sieve did not eliminate enough possibilities for $m \pmod {N_l}$ and we could not reach a contradiction. If we do not sieve using the first method of elimination, then we are unable to eliminate the possibility that $P$ reduces to a cusp modulo $l$ for each prime in our list (i.e. that each prime is a prime of multiplicative reduction for $E$), and so we will not obtain a contradiction.

The sieving method for $X_0(43)$ is identical. The curve is of genus $3$ and we used the following plane quartic model in $\mathbb{P}^3$: \begin{align*}
64X^4 + 48X^3Y + 16X^2Y^2 + 8XY^3 - 3Y^4+ \\ (16X^2+8XY+2Y^2)T^2+T^4 = 0, 
\end{align*} with the Atkin--Lehner involution given by $(x:y:t) \mapsto (x:y: -t)$, and the map $\psi$ to $X_0^+(43) \in \mathbb{P}(1,1,2)$ given by $(x:y:t) \mapsto (x:y:t)$. We applied the sieve for $d = 74$. We obtained a contradiction using the primes $3,5,7,17,19,29,31,47,59,61,71,73,79,107$.

\section{Hecke Operators and Hilbert Newforms}

\subsection{Bounding the Exponent}
Once we have obtained irreducibility of the mod-$p$ Galois representations of our Frey curve, the next step is to apply the level-lowering theorem (Theorem \ref{levellower}). By our previous work, we have a list of possible levels $\mathcal{N}_p$ for our Hilbert newform, $\mathfrak{f}$, which are displayed in the appendix. We consider each possibility separately, and aim to discard all isomorphisms between the representations of our Frey curve and the newforms at this level. If we can do this at all the possible levels then we will obtain our desired contradiction.

The standard idea, as used in \citep[p.~12]{realquad}, is as follows: compute the newforms at the level $\mathcal{N}_p$ and combine local information mod $\mathfrak{q}$ for many primes to obtain a contradiction, one newform at a time. For $\mathfrak{q}$ a prime of $K$ not dividing $\mathcal{N}_p$, recall from Section 3 the notation \[ \mathcal{A}_\mathfrak{q} := \{ a \in \mathbb{Z} : \abs{a} \leq 2 \sqrt{n_\mathfrak{q}}, \quad n_\mathfrak{q}+1 - a \equiv 0 \pmod{4}  \}. \] If $\mathfrak{q}$ is a prime of good reduction for $E$, then as discussed in Section 3, $a_{\mathfrak{q}}(E) \in \mathcal{A}_\mathfrak{q}$. 

The following lemma gives the standard method of bounding the prime $p$. We use the same notation as in Section 2.

\begin{lemma}[{\citep[p.~12]{realquad}}]\label{Cf1} Suppose $\overline{\rho}_{E,p} \sim \overline{\rho}_{\mathfrak{f},\varpi}$. Let $\mathcal{T}$ be a set of prime ideals $\mathfrak{q}$ which do not divide $\mathcal{N}_p$. For each $\mathfrak{q} \in \mathcal{T}$ define the principal ideal \[B_{\mathfrak{f},\mathfrak{q}}:= \big( n_{\mathfrak{q}}(n_{\mathfrak{q}}+1-a_{\mathfrak{q}}\mathfrak{f})(n_{\mathfrak{q}}+1+a_{\mathfrak{q}}\mathfrak{f}) \prod_{a \in \mathcal{A}_\mathfrak{q}} (a-a_{\mathfrak{q}}\mathfrak{f}) \big) \cdot  \mathcal{O}_{\Q_\mathfrak{f}}.\] Set $B_\mathfrak{f}:= \sum_{\mathfrak{q} \in T} B_{\mathfrak{f},\mathfrak{q}}$, and denote by $C_\mathfrak{f}$ the norm of this ideal. Then $p \mid C_\mathfrak{f}$.
\end{lemma}

If $C_\mathfrak{f}$ is non-zero we obtain a bound on $p$. If all the prime factors of $C_\mathfrak{f}$ are less than $17$, then this discards the isomorphism for all $p$ we are concerned with. We discuss the case $C_\mathfrak{f}=0$ at the end of this section.

For the levels $\mathcal{N}_p$ appearing in \citep{realquad} (i.e. when $2 \leq d \leq 23$), this method works as we can compute the newforms, but for larger levels this is not possible with the current Magma implementation, as discussed in \citep{realquad}. The aim of this section is to provide a work-around for this by working directly with Hecke operators. This is similar to what was done in \citep{superelliptic} where the levels obtained were too large to compute the newforms. By working directly with Hecke operators we will be able to reconstruct the eigenvalues of the newforms for the primes $\mathfrak{q} \in \mathcal{T}$. By doing this, we often lose out on knowing what the Hecke eigenfields, $\Q_\mathfrak{f}$, are, and so computing the norm of a sum of ideals in $\mathcal{O}_{\Q_\mathfrak{f}}$ as in Lemma \ref{Cf1} is impossible. The following lemma addresses this issue.

\begin{lemma}\label{cf2} Suppose $\overline{\rho}_{E,p} \sim \overline{\rho}_{\mathfrak{f},\varpi}$. Let $\mathcal{T}$ be a set of prime ideals $\mathfrak{q}$ which do not divide $\mathcal{N}_p$. For each $\mathfrak{q}$, define the field $L_{\mathfrak{q}}:= \Q(a_\mathfrak{q}\mathfrak{f})$, and define the element \[ b_{\mathfrak{f},\mathfrak{q}}:=  n_{\mathfrak{q}}(n_{\mathfrak{q}}+1-a_{\mathfrak{q}}\mathfrak{f})(n_{\mathfrak{q}}+1+a_{\mathfrak{q}}\mathfrak{f}) \prod_{a \in \mathcal{A}} (a-a_{\mathfrak{q}}\mathfrak{f}) \in L_{\mathfrak{q}}. \] Let $c_{\mathfrak{f},\mathfrak{q}}:=\mathrm{Norm}_{L_{\mathfrak{q}} / \Q} (b_{\mathfrak{f},\mathfrak{q}})$. Write $c_\mathfrak{f}:=\gcd \{ c_{\mathfrak{f},\mathfrak{q}} : \mathfrak{q} \in T \} $. Then $p \mid c_\mathfrak{f}$.
\end{lemma}

\begin{proof} 
We start by noting that the Hecke eigenfield $\Q_\mathfrak{f}$ of $\mathfrak{f}$ contains $L_{\mathfrak{q}}$ for each $\mathfrak{q} \in \mathcal{T}$, and we view $L_{\mathfrak{q}}$ as a subfield of $\Q_\mathfrak{f}$.
Following the notation of the previous lemma, we note that $b_{\mathfrak{f},\mathfrak{q}} \in B_\mathfrak{f}$ for all $\mathfrak{q} \in \mathcal{T}$. The norm of an ideal is the greatest common divisor of the norm of all of its elements, and so \[ p \quad \mid \quad C_\mathfrak{f} \quad | \quad  \gcd \{ \mathrm{Norm}_{\Q_\mathfrak{f}/\Q}(b_{\mathfrak{f},\mathfrak{q}}) : \mathfrak{q} \in\mathcal{T}\}. \] As $L_\mathfrak{q}$ is a subfield of $\Q_\mathfrak{f}$ we have that \[ \mathrm{Norm}_{\Q_\mathfrak{f}/\Q}(b_{\mathfrak{f},\mathfrak{q}}) = \left( \mathrm{Norm}_{L_{\mathfrak{q}} / \Q} (b_{\mathfrak{f},\mathfrak{q}}) \right) ^ { [\Q_\mathfrak{f}: L_{\mathfrak{q}}]}, \] so it follows that $p \mid c_\mathfrak{f}$.
\end{proof}

The prime factors of $C_\mathfrak{f}$ are contained in the set of prime factors of $c_\mathfrak{f}$, and so this version may give worse bounds, but we found that in practice, by considering enough primes $\mathfrak{q}$, the two sets of prime factors coincide.

\begin{remark} We in fact work directly with the minimal polynomial of $a_\mathfrak{q}(\mathfrak{f})$ over $\Q$, which we denote $\mu$, to define $c_{\mathfrak{f},\mathfrak{q}}$. We have \[c_{\mathfrak{f},\mathfrak{q}} = n_{\mathfrak{q}} \cdot  \mu(n_{\mathfrak{q}}+1) \cdot \mu(-n_{\mathfrak{q}}-1) \cdot \prod_{a \in \mathcal{A}} \mu(a).  \]
\end{remark}

\subsection{Reconstructing Hilbert Newforms}

Let $\mathfrak{q}$ be a prime of $K$. Write $T_{\mathfrak{q}}$ for the Hecke operator on the space of newforms at level $\mathcal{N}_p$, and write $\chi_\mathfrak{q}$ for its characteristic polynomial. We view $T_\mathfrak{q}$ as a matrix. We then have a factorisation into irreducible polynomials  \[ \chi_\mathfrak{q}(X) = \prod_{i=1}^r e_{\mathfrak{q},i}(X)^{m_i}. \] The roots of each irreducible factor $e_i$ are the eigenvalues of a Galois conjugacy class of Hilbert newforms. Associated to each $e_i$, we have the corresponding irreducible subspace \[ V_{\mathfrak{q},i} := \ker(e_{\mathfrak{q},i}(T_\mathfrak{q})),\] with a basis consisting of members of Galois conjugacy classes of newforms whose eigenvalues at $\mathfrak{q}$ satisfy $e_{\mathfrak{q},i}$. To each $e_{\mathfrak{q},i}$, we also associate a value $c_{\mathfrak{q},i} := c_{\mathfrak{f},\mathfrak{q}}$ for any newform $\mathfrak{f}$ with eigenvalue at $\mathfrak{q}$ a root of $e_{\mathfrak{q},i}$.

Since the Hecke operators $T_\mathfrak{q}$ commute, if $T_{\mathfrak{q}_1}$ and $T_{\mathfrak{q}_2}$ are Hecke operators, and $V_{\mathfrak{q}_1,i_1}$ is some irreducible subspace with respect to $T_{\mathfrak{q}_1}$, then it is preserved by $T_{\mathfrak{q}_2}$ and we can compute the matrix of $T_{\mathfrak{q}_2}$ restricted to $V_{\mathfrak{q}_1,i_1}$. We can then compute the characteristic polynomial of this matrix and decompose $V_{\mathfrak{q}_1,i_1}$ into irreducible subspaces under $T_{\mathfrak{q}_2}$, which we denote by $V_{\mathfrak{q}_1,\mathfrak{q}_2,i_1,i_2}$. Such a  subspace will have a basis of newforms whose eigenvalues at $\mathfrak{q}_1$ and $\mathfrak{q}_2$ are roots of $e_{\mathfrak{q}_1,i_1}$ and $e_{\mathfrak{q}_2,i_2}$ respectively. Associated to the subspace  $V_{\mathfrak{q}_1,\mathfrak{q}_2,i_1,i_2}$ is the integer $c_{\mathfrak{q}_1,\mathfrak{q}_2,i_1,i_2} := \gcd(c_{\mathfrak{f},\mathfrak{q}_1},c_{\mathfrak{f},\mathfrak{q}_2}) $, where $\mathfrak{f}$ is any newform in the space $V_{\mathfrak{q}_1,\mathfrak{q}_2,i_1,i_2}$.

We continue this process. If a value $c_{\mathfrak{q}_1 \cdots \mathfrak{q}_m, i_1, \cdots, i_m}$ has all its prime factors $\leq 13$ then we can discard the associated subspace, as we know that $\overline{\rho}_{E,p} \not\sim \overline{\rho}_{\mathfrak{f},\varpi}$ for any newform in this subspace. We aim to discard all possible subspaces, hence obtaining a contradiction. 

We carried out this process for values $d$ for which the maximum dimension of the space of newforms is $< 9000$ (see the table in the appendix) and this proves Theorem \ref{Mainthm}. The maximum dimension we considered was $8960$ in the case $d=66$. For dimensions larger than this, we found computing the Hecke operators to be computationally impractical.

By considering enough primes $\mathfrak{q}$, we found that in the cases where we \emph{can} compute the full newform decomposition, we were able to completely reconstruct the data using the method described above. We also verified that our results agree with those in \citep{realquad} when $d<26$. We could usually eliminate all subspaces at each level. For the values $d=33,34,41,55,57,89$, we obtained a value $c_\mathfrak{f} = 0$. We consider these cases in Section 5.3. Also, for the pairs $(d,p) = (34,23), (55,23), (86,31),$ and $(97,17)$, we obtained values $c_\mathfrak{f} $ divisible by $p$, which is why these values of $p$ appear in the statement of Theorem \ref{thm2} (but not in Table 1), as we were unable to discard these isomorphisms.

\subsection{Remaining Cases}

In some cases we can discard an isomorphism, or discard it for certain primes, even if $c_\mathfrak{f} = 0$ or has a prime factor $\geq 17$. We first consider the following image of inertia argument.

\begin{lemma}[Image of Inertia {\citep[p.~13]{realquad}}] 
Suppose $\mathfrak{f}$ is a Hilbert newform with $\Q_\mathfrak{f} = \Q$, and write $E'$ for the elliptic curve associated to $\mathfrak{f}$. Suppose that one of $E$ and $E'$ has potentially multiplicative reduction at a prime $\mathfrak{q}$ and that the other has potentially good reduction at $\mathfrak{q}$. Then $\overline{\rho}_{E,p} \not\sim \overline{\rho}_{\mathfrak{f},\varpi}$.
\end{lemma}

We applied this argument when $d=34$ and $d=55$, with $\mathcal{N}_p= \mathfrak{p}^8$ in each case, using the (unique) prime above $2$ which is of potentially multiplicative reduction for $E$ (Lemma \ref{reduc}), but of potentially good reduction for each elliptic curve corresponding to a rational newform with $c$-value $0$. This completes the proof of Theorem \ref{thm2}.

Note that from a rational newform $\mathfrak{f}$ (or equivalently an irreducible subspace of dimension 1) we can obtain the corresponding elliptic curve $E'$ as follows. Using the \texttt{EllipticCurveSearch} function in \texttt{Magma} we obtain a (potentially incomplete) list of elliptic curves with conductor $\mathcal{N}_p$. We see if we can find a curve $E'$ such that $a_\mathfrak{q}(E') = a_\mathfrak{q}\mathfrak{f}$, say for a few primes $\mathfrak{q}$. If the values $a_\mathfrak{q}(E')$ do not equal $a_\mathfrak{q} \mathfrak{f}'$ for any other newform $\mathfrak{f}'$, then by modularity $E'$ must correspond to $\mathfrak{f}$. Even when we cannot compute the full newform decomposition, we can still verify this by considering the eigenvalues associated to each subspace (obtained using the method described in Section 5.2). 

We can also often deal with fixed values of $p$, and hence obtain a bound on $p$, using a method of Kraus. The following lemma is stated in \citep[p.~2]{sqrt5} for $K = \Q(\sqrt{5})$, but is easily generalised to $\Q(\sqrt{d})$. It is based on knowing certain primes of multiplicative reduction for $E$ (see Lemma \ref{krauslem}).

\begin{lemma}[{\citep[p.~2]{sqrt5}}]
Let $p \geq 17$ be a prime and suppose there exists a natural number $n$ satisfying the following conditions:
\begin{itemize}
\item we have $n < p-2$ and $n \equiv 2 \pmod{4}$;
\item we have $q:=np+1$ is a prime that splits in $\mathcal{O}_K$;
\item we have $q \nmid \mathrm{Res}(X^n-1,(X+1)^n-1)$.
\end{itemize} Then $\overline{\rho}_{E,p} \not\sim \overline{\rho}_{\mathfrak{f},\varpi}$ for any rational newform $\mathfrak{f}$.
\end{lemma}

We apply this lemma in the cases $d = 33,34,41,57,89$, as well as for $d=17$, to show that $\overline{\rho}_{E,p} \not\sim \overline{\rho}_{\mathfrak{f},\varpi}$ for all but finitely many $p \leq 10^7$ when $c_\mathfrak{f}=0$. We then remove these leftover primes by choosing $n$ appropriately, as in \citep[pp.~10-11]{sqrt5}. We were able to do this for each leftover prime other than $p=19$ in the case $d=57$, which we consider in Section 6. 

This strategy does not help eliminate the irrational newforms whose $c$-values are divisible by  a prime $p$ in the cases appearing in Theorem \ref{thm2}, as the prime $p$ is appearing as a factor of $(n_{\mathfrak{q}}+1-a_{\mathfrak{q}}\mathfrak{f})(n_{\mathfrak{q}}+1+a_{\mathfrak{q}}\mathfrak{f})$ for each $\mathfrak{q}$, and so using primes of multiplicative reduction will not help us rule it out. When working over $\Q$, the standard strategy at this point is to apply an argument using a Sturm bound. Although Sturm bounds do exist for Hilbert newforms over quadratic fields (see \citep{sturmquad}), they are too large to be of use computationally in these cases (the bound would be much larger than the dimensions of the spaces of newforms, which would already be too large to compute with). We refer to \citep{sturm} for similar discussions around Sturm bounds, and for other techniques which may be used for eliminating newforms. It may be possible to adapt these methods to this setting.

Finally, in the case $d=89$, we apply the result of \citep[p.~13]{realquad} (in the same way it was applied in the case $d=17$ in \citep[p.~13]{realquad}) to conclude that we have no solutions if $p \equiv \pm 2 \pmod{5}$.

These results prove Theorem \ref{thm3}, apart from the case  $d=57, p = 19$ which we deal with in the next section.

\section{Regular Primes for Quadratic Fields}

In this section we see how we can sometimes avoid using the modular method altogether to show that we have no solutions to the Fermat equation over real quadratic fields for certain primes. We also complete the proof of Theorem \ref{thm3} by showing that the Fermat equation over $\Q(\sqrt{d})$ has no non-trivial solutions for $p=19$ when $d=57$. We note in passing that the methods of this section are unsuccessful for most values of $p$ and $d$ appearing in Table 1.

A prime $p$ is said to be \emph{regular} if it does not divide the class number of the cyclotomic extension $\Q(\zeta_p)$, and \emph{irregular} otherwise. Extending this notion, for $d'$ a squarefree integer, a prime $p$ is said to be \emph{$d'$-regular} if it does not divide the class number of $\Q(\sqrt{d'},\zeta_p)$, and \emph{$d'$-irregular} otherwise. If $p$ is an irregular prime, then $p$ is also $d'$-irregular for all $d'$.

\begin{theorem} Let $p \geq 5$ be a $d$-regular prime, with $d>0$. If $p \nmid d$ then the Fermat equation with exponent $p$ has no non-trivial solutions in $\Q(\sqrt{d})$. If $d = p \cdot m$, then the Fermat equation with exponent $p$ has no non-trivial solutions in $\Q(\sqrt{d})$ if $-m$ is a square mod $p$.
\end{theorem}

\begin{proof}
Let $p \geq 5$ be a $d$-regular prime with $d>0$. Suppose $p \nmid d$. If $d$ is a square mod $p$, the result holds by \citep[p.~129]{reg}. If $d$ is not a square mod $p$, then combining the results of \citep[p.~126,129]{reg} and \citep[p.~2]{krausreg} shows there are no non-trivial solutions. In the case that $p \mid d$, we again apply the result of \citep[p.~129]{reg} to conclude.
\end{proof}

To complete the proof of Theorem \ref{thm3} it therefore suffices to show that $19$ is $57$-regular (as $-3$ is a square mod $19$). In general, directly computing the class numbers of cyclotomic extensions of quadratic fields is not possible, but using the work of Hao and Parry on generalised Bernoulli numbers \citep{bernoulli}, we can avoid doing this. We note that $\Q(\sqrt{57},\zeta_{19}) = \Q(\sqrt{-3},\zeta_{19})$, and so it is equivalent to show that $19$ is $-3$-regular. The tables in \citep{bernoulli} show that this is indeed the case. This completes the proof of Theorem \ref{thm3}. 

For completeness, we state the following result, which gives a simple criterion, when $p \nmid d$ and $d>0$, to check if a prime is $d$-regular.

\begin{proposition}[{\citep[p.~276]{bernoulli}}] Let $p$ be an odd regular prime, let $d>0$, and suppose $p \nmid d$. Write $\Delta$ for the discriminant of $\Q(\sqrt{d})$. Then $p$ is $d$-regular if and only if for all odd $n$ with $1 \leq n \leq p-2$,  \[ \sum_{j=1}^p S_n(j)A_{j \Delta} \not\equiv 0 \pmod{p}. \] Here, \[ S_n(j) = \sum_{u=0}^{j-1} u^n \quad \text{and} \quad  A_{j\Delta} = \sum_{\substack{t=1 \\ t \equiv j\Delta \pmod{p}}}^\Delta \left(\frac{\Delta}{t} \right), \] where $\left(\frac{\Delta}{t} \right)$ denotes the Kronecker symbol.
\end{proposition}

We note that when $p$ is a regular prime satisfying $p > \Delta$, a similar criterion to test whether $p$ is $d$-regular is given in \citep[p.~279]{bernoulli} which is faster to check computationally. We found, using a short \texttt{Magma} script, that $p$ is $d$-irregular for each pair $(d,p)$ appearing in Theorem \ref{thm2}, which is why we cannot eliminate these primes. We also checked that the $2$- and $5$-irregular primes obtained using our code agree with those appearing in the tables in  \citep{bernoulli}. 

\section*{Appendix}

In the table below, $n$ and $n_{\mathrm{new}}$ denote the dimensions of the spaces of Hilbert cuspforms and Hilbert newforms respectively. The column denoted RCG records the exponents of the ray class groups appearing in Lemma \ref{rcg}. The remaining column headings follow the notation of the paper.

\tiny
\begingroup
\renewcommand*{\arraystretch}{1.4}
\begin{longtable}{|c|c|c|c|c|c|c|c|}
\hline
$d$ & $S$ & $r$ & $\mathfrak{m}$ & $\mathcal{N}_p$ & $n$ &$n_{\mathrm{new}}$ &   RCG  \\
\hline 
\hline
$26$ & $\mathfrak{p}$ & $2$ & $\langle 5, \sqrt{d}+6 \rangle$ & $\mathfrak{p}$ & $18$ & $2$ &  $8$  \\ \cline{5-8} & & & & $\mathfrak{m}^2  \mathfrak{p}$ & $388$ & $78$ &  \\
\hline
$29$ & $ \mathfrak{p} = 2\mathcal{O}$ & $1$ & $1$ & $\mathfrak{p}$ & $3$ & $1$ &  $2$  \\  \cline{5-8} & & & & $\mathfrak{p}^4$ & $81$ & $45$ & \\
\hline
$30$ & $ \mathfrak{p}$ & $2$ & $\langle 3, \sqrt{d}\rangle$ & $\mathfrak{p}$ & $28$& $0$ &  $2,4$  \\ \cline{5-8} & & &  & $\mathfrak{m}^2  \mathfrak{p}$ & $220$ & $28$ &  \\ \cline{5-8} & & & & $\mathfrak{p}^8 $ & $2172$ & $544$ &    \\ \cline{5-8} & & &  & $\mathfrak{m}^2  \mathfrak{p}^8$ & $26108$ & $2720$ &    \\
\hline
$31$ & $\mathfrak{p}$ & $1$ & $1$ & $\mathfrak{p}$ & $16$ & $2$ &  $2$ \\ \cline{5-8} & & & & $\mathfrak{p}^4$ & $93$ & $20$ &  \\
\hline
$33$ &  $\mathfrak{p}_1,\mathfrak{p}_2 $ & $1$ & $1$ & $\mathfrak{p}_1\mathfrak{p}_2$ & $6$ & $2$ &  $2$  \\ \cline{5-8} & & &  & $\mathfrak{p}_1  \mathfrak{p}_2^4$ & $34$ & $2$ &  \\
\hline
$34$ & $ \mathfrak{p}$ & $2$ & $\langle 3,\sqrt{d}+1 \rangle$ & $\mathfrak{p}$ & $36$& $4$ &  $4,8$ \\ \cline{5-8} & & &  & $\mathfrak{m}^2  \mathfrak{p}$ & $292$ & $40$ &  \\ \cline{5-8} & & & & $\mathfrak{p}^8 $ & $2940$ & $736$ &    \\ \cline{5-8} & & &  & $\mathfrak{m}^2  \mathfrak{p}^8$ & $35324$ & $3680$ &    \\
\hline
$35$ & $ \mathfrak{p}$ & $2$ & $\langle 5,\sqrt{d} \rangle$ &  $\mathfrak{p}$ & $28$& $0$ &  $2,8$ \\ \cline{5-8} & & &  & $\mathfrak{m}^2  \mathfrak{p}$ & $592$ & $120$ &   \\ \cline{5-8} & & & & $\mathfrak{p}^4 $ & $160$ & $38$ &     \\ \cline{5-8} & & &  & $\mathfrak{m}^2  \mathfrak{p}^4$ & $4580$ & $722$ &   \\  
\hline
$37$  & $ \mathfrak{p} = 2\mathcal{O}$ & $1$ & $1$ & $\mathfrak{p}$ & $4$ & $2$ & $6$ \\  \cline{5-8} & & & & $\mathfrak{p}^4$ & $135$ & $75$ &   \\
\hline
$38$ & $\mathfrak{p}$ & $1$ & $1$ & $\mathfrak{p}$ & $18$ & $0$ &  $4$   \\ \cline{5-8} & & & & $\mathfrak{p}^8$ & $1310$ & $328$ &  \\
\hline 
$39$ & $\mathfrak{p}$ & $2$ & $\langle 5,\sqrt{d} +3\rangle$ & $\mathfrak{p}$  & $36$ & $4$ &  $2,4,8$ \\ \cline{5-8} & & & & $\mathfrak{p}^4$ & $236$ & $56$ &    \\ \cline{5-8} & & & & $\mathfrak{m}^2\mathfrak{p}$ & $792$ & $156$ &   \\ \cline{5-8} & & & & $\mathfrak{p}^8$ & $3356$ & $832$  &  \\ \cline{5-8} & & & & $\mathfrak{m}^2\mathfrak{p}^4$ & $6284$ & $984$ &  \\ \cline{5-8} & & & & $\mathfrak{m}^2\mathfrak{p}^8$ & $99900$ & $15808$ &    \\
\hline
$41$ & $\mathfrak{p}_1,\mathfrak{p}_2 $ & $1$ & $1$ & $\mathfrak{p}_1\mathfrak{p}_2$ & $6$ & $2$ &  $1$ \\
\hline
$42$  & $ \mathfrak{p}$ & $2$ &  $\langle 3,\sqrt{d}\rangle$ & $\mathfrak{p}$ & $36$& $4$ &  $4$  \\ \cline{5-8} & & &  & $\mathfrak{m}^2  \mathfrak{p}$ & $320$ & $36$ &   \\ \cline{5-8} & & & & $\mathfrak{p}^8 $ & $3484$ & $864$ &  \\ \cline{5-8} & & &  & $\mathfrak{m}^2  \mathfrak{p}^8$ & $41468$ & $4320$ &     \\
\hline
$43$  & $\mathfrak{p}$ & $1$ & $1$ & $\mathfrak{p}$ & $20$ & $0$ &  $2$  \\ \cline{5-8} & & & & $\mathfrak{p}^4$ & $127$ & $33$ &  \\
\hline
$46$ & $\mathfrak{p}$ & $1$ & $1$ & $\mathfrak{p}$ & $26$ & $4$ &  $4$ \\ \cline{5-8} & & & & $\mathfrak{p}^8$ & $2390$ & $592$ &  \\ 
\hline
$47$  & $\mathfrak{p}$ & $1$ & $1$ & $\mathfrak{p}$ & $24$ & $2$ &  $2$  \\ \cline{5-8} & & & & $\mathfrak{p}^4$ & $135$ & $28$ & \\ 
\hline
$51$  & $ \mathfrak{p}$ & $2$ & $\langle 3,\sqrt{d}\rangle$ & $\mathfrak{p}$ & $52$ & $4$ &  $2,4$ \\ \cline{5-8} & & &  & $\mathfrak{m}^2  \mathfrak{p}$ & $468$ & $56$ &  \\ \cline{5-8} & & & & $\mathfrak{p}^4 $ & $320$ & $84$ &    \\ \cline{5-8} & & &  & $\mathfrak{m}^2  \mathfrak{p}^4$ & $3740$ & $396$ &   \\
\hline
$53$  & $ \mathfrak{p} = 2\mathcal{O}$ & $1$ & $1$ & $\mathfrak{p}$ & $6$ & $2$ &  $2$  \\  \cline{5-8} & & & & $\mathfrak{p}^4$ & $189$ & $105$ &   \\ 
\hline
$55$ & $\mathfrak{p}$ & $2$ & $\langle 3,\sqrt{d}+4 \rangle$ & $\mathfrak{p}$  & $68$ & $12$ &  $2,4,8$   \\ \cline{5-8} & & & & $\mathfrak{p}^4$ & $412$ & $96$ &    \\ \cline{5-8} & & & & $\mathfrak{m}^2\mathfrak{p}$ & $584$ & $84$ &   \\ \cline{5-8} & & & & $\mathfrak{p}^8$ & $5916$ & $1472$ &     \\ \cline{5-8} & & & & $\mathfrak{m}^2\mathfrak{p}^4$ & $4460$ & $464$ &     \\ \cline{5-8} & & & & $\mathfrak{m}^2\mathfrak{p}^8$ & $70652$ & $7360$ &     \\ 
\hline
$57$  &  $\mathfrak{p}_1,\mathfrak{p}_2 $ & $1$ & $1$ & $\mathfrak{p}_1\mathfrak{p}_2$ & $12$ & $4$ &  $2$  \\ \cline{5-8} & & &  & $\mathfrak{p}_1  \mathfrak{p}_2^4$ & $82$ & $6$ & \\
\hline
$58$ & $\mathfrak{p}$ & $2$ & $\langle 3,\sqrt{d}+1 \rangle$ & $\mathfrak{p}$ & $50$ & $10$ &  $4$  \\ \cline{5-8} & & & & $\mathfrak{m}^2  \mathfrak{p}$ & $592$ & $90$ &   \\
\hline
$59$  & $\mathfrak{p}$ & $1$ & $1$ & $\mathfrak{p}$ & $32$ & $4$ &  $2$ \\ \cline{5-8} & & & & $\mathfrak{p}^4$ & $177$ & $47$ &  \\ 
\hline
$61$  & $ \mathfrak{p} = 2\mathcal{O}$ & $1$ & $1$ & $\mathfrak{p}$ & $7$ & $3$ &  $2$  \\  \cline{5-8} & & & & $\mathfrak{p}^4$ & $295$ & $165$ &   \\
\hline
$62$ & $\mathfrak{p}$ & $1$ & $1$ & $\mathfrak{p}$ & $32$ & $2$ &  $4$\\ \cline{5-8} & & & & $\mathfrak{p}^8$ & $2710$ & $672$ &  \\ 
\hline
$65$  &  $\mathfrak{p}_1,\mathfrak{p}_2 $ & $2$ & $\langle 7,\sqrt{d}+3\rangle$ & $\mathfrak{p}_1\mathfrak{p}_2$ & $24$ & $8$ &  $4$  \\ \cline{5-8} & & &  & $\mathfrak{m}^2\mathfrak{p}_1  \mathfrak{p}_2$ & $722$ & $54$ &   \\ 
\hline
$66$  & $ \mathfrak{p}$ & $2$ & $\langle 3,\sqrt{d}\rangle$ & $\mathfrak{p}$ & $76$& $8$ &  $2,4$\\ \cline{5-8} & & &  & $\mathfrak{m}^2  \mathfrak{p}$ & $688$ & $88$ &   \\ \cline{5-8} & & & & $\mathfrak{p}^8 $ & $7164$ & $688$ &   \\ \cline{5-8} & & &  & $\mathfrak{m}^2  \mathfrak{p}^8$ & $86012$ & $8960$ &    \\ 
\hline
$67$ & $ \mathfrak{p}$ & $1$ & $1$ & $\mathfrak{p}$ & $36$ & $4$ &  $2$  \\  \cline{5-8} & & & & $\mathfrak{p}^4$ & $247$ & $63$ &   \\  
\hline
$69$  & $ \mathfrak{p} = 2\mathcal{O}$ & $1$ & $1$ & $\mathfrak{p}$ & $10$ & $4$ &  $2$  \\  \cline{5-8} & & & & $\mathfrak{p}^4$ & $330$ & $177$ &   \\
\hline
$70$ & $ \mathfrak{p}$ & $2$ & $\langle 3,\sqrt{d}+1\rangle$ & $\mathfrak{p}$ & $88$& $16$ & $2,4$ \\ \cline{5-8} & & &  & $\mathfrak{m}^2  \mathfrak{p}$ & $840$ & $120$ &  \\ \cline{5-8} & & & & $\mathfrak{p}^8 $ & $8572$ & $2144$ &        \\ \cline{5-8} & & &  & $\mathfrak{m}^2  \mathfrak{p}^8$ & $102908$ & $10720$ &    \\  
\hline
$71$  & $ \mathfrak{p}$ & $1$ & $1$ & $\mathfrak{p}$ & $42$ & $8$ &  $2$ \\  \cline{5-8} & & & & $\mathfrak{p}^4$ & $265$ & $58$ & \\   
\hline
$73$ & $\mathfrak{p}_1,\mathfrak{p}_2 $ & $1$ & $1$ & $\mathfrak{p}_1\mathfrak{p}_2$ & $16$ & $4$ &  $1$ \\ 
\hline
$74$  & $\mathfrak{p}$ & $2$ & $\langle 5,\sqrt{d}+3\rangle$ & $\mathfrak{p}$ & $72$ & $12$ &  $4$ \\ \cline{5-8} & & & & $\mathfrak{m}^2  \mathfrak{p}$ & $1868$ & $384$ &   \\
\hline
$77$  & $ \mathfrak{p} = 2\mathcal{O}$ & $1$ & $1$ & $\mathfrak{p}$ & $10$ & $4$ &  $2$  \\  \cline{5-8} & & & & $\mathfrak{p}^4$ & $330$ & $177$ &   \\ 
\hline
$78$ & $ \mathfrak{p}$ & $2$ & $\langle 7,\sqrt{d}+6\rangle$ & $\mathfrak{p}$ & $88$& $8$ &  $2,4$ \\ \cline{5-8} & & &  & $\mathfrak{m}^2  \mathfrak{p}$ & $3896$ & $960$ &   \\ \cline{5-8} & & & & $\mathfrak{p}^8 $ & $8828$ & $2208$ &  \\ \cline{5-8} & & &  & $\mathfrak{m}^2  \mathfrak{p}^8$ & $494588$ & $90528$ &    \\  
\hline
$79$ & $ \mathfrak{p}$ & $1$ & $1$ & $\mathfrak{p}$ & $156$ & $30$ &  $6$  \\  \cline{5-8} & & & & $\mathfrak{p}^4$ & $1077$ & $252$ &   \\  
\hline
$82$  & $\mathfrak{p}$ & $2$ & $\langle 3,\sqrt{d}+4\rangle$ & $\mathfrak{p}$ & $168$ & $40$ &  $8$  \\ \cline{5-8} & & & & $\mathfrak{m}^2  \mathfrak{p}$ & $1940$ & $284$ &  \\ 
\hline
$83$ & $ \mathfrak{p}$ & $1$ & $1$ & $\mathfrak{p}$ & $44$ & $4$ &  $2$  \\  \cline{5-8} & & & & $\mathfrak{p}^4$ & $265$ & $69$ & \\  
\hline
$85$& $ \mathfrak{p}$ & $2$ & $\langle 3,\sqrt{d}+5\rangle$ & $\mathfrak{p}$ & $22$& $10$ &  $2,4$  \\ \cline{5-8} & & &  & $\mathfrak{m}^2  \mathfrak{p}$ & $178$ & $44$ &   \\ \cline{5-8} & & & & $\mathfrak{p}^4 $ & $966$ & $540$ &  \\ \cline{5-8} & & &  & $\mathfrak{m}^2  \mathfrak{p}^4$ & $11518$ & $2700$ &   \\ 
\hline
$86$ & $\mathfrak{p}$ & $1$ & $1$ & $\mathfrak{p}$ & $50$ & $8$ &  $4$  \\ \cline{5-8} & & & & $\mathfrak{p}^8$ & $4958$ & $1240$ & \\ 
\hline
$87$ & $ \mathfrak{p}$ & $2$ & $\langle 3,\sqrt{d}\rangle$ & $\mathfrak{p}$ & $88$& $16$ &  $2,4$ \\ \cline{5-8} & & &  & $\mathfrak{m}^2  \mathfrak{p}$ & $932$ & $116$ &   \\ \cline{5-8} & & & & $\mathfrak{p}^4 $ & $656$ & $162$ &    \\ \cline{5-8} & & &  & $\mathfrak{m}^2  \mathfrak{p}^4$ & $7484$ & $786$ &  \\ 
\hline
$89$ & $\mathfrak{p}_1,\mathfrak{p}_2 $ & $1$ & $1$ & $\mathfrak{p}_1\mathfrak{p}_2$ & $20$ & $4$ &  $1$  \\
\hline
$91$ & $ \mathfrak{p}$ & $2$ & $\langle 5,\sqrt{d}+1\rangle$ & $\mathfrak{p}$ & $120$& $20$ &  $2,4,8$ \\ \cline{5-8} & & &  & $\mathfrak{m}^2  \mathfrak{p}$ & $3128$ & $660$ &  \\ \cline{5-8} & & & & $\mathfrak{p}^4 $ & $832$ & $206$ &  \\ \cline{5-8} & & &  & $\mathfrak{m}^2  \mathfrak{p}^4$ & $24740$ & $3914$ &  \\
\hline
$93$  & $ \mathfrak{p} = 2\mathcal{O}$ & $1$ & $1$ & $\mathfrak{p}$ & $14$ & $6$ &  $2$  \\  \cline{5-8} & & & & $\mathfrak{p}^4$ & $330$ & $177$ & \\
\hline
$94$ & $\mathfrak{p}$ & $1$ & $1$ & $\mathfrak{p}$ & $68$ & $14$ &  $4$  \\ \cline{5-8} & & & & $\mathfrak{p}^8$ & $6822$ & $1696$ & \\ 
\hline
$95$ & $\mathfrak{p}$ & $2$ & $\langle7,\sqrt{d}+2\rangle$ & $\mathfrak{p}$  & $116$ & $20$ & $2,4,8$  \\ \cline{5-8} & & & & $\mathfrak{p}^4$ & $756$ & $180$ &   \\ \cline{5-8} & & & & $\mathfrak{m}^2\mathfrak{p}$ & $4848$ & $1188$ & \\ \cline{5-8} & & & & $\mathfrak{p}^8$ & $11068$ & $2752$ &     \\ \cline{5-8} & & & & $\mathfrak{m}^2\mathfrak{p}^4$ & $38572$ & $7060$ &     \\ \cline{5-8} & & & & $\mathfrak{m}^2\mathfrak{p}^8$ & $616444$ & $112832$ &   \\ 
\hline
$97$ & $\mathfrak{p}_1,\mathfrak{p}_2 $ & $1$ & $1$ & $\mathfrak{p}_1\mathfrak{p}_2$ & $25$ & $4$ &  $1$  \\ 
\hline

\end{longtable}
\endgroup

\Addresses

\end{document}